\documentclass[10pt]{amsart}
\usepackage{latexsym}
\usepackage{amsmath}
\usepackage{amssymb}
\usepackage{mathrsfs}
\usepackage{graphicx}
\usepackage{color}
\usepackage{pgfpages}
\usepackage{ifthen}
\usepackage{leftidx,tensor}
\usepackage[T1]{fontenc}
\usepackage[latin1]{inputenc}
\usepackage{mathtools}
\usepackage{comment}
\usepackage{dsfont}
\usepackage[nocompress]{cite}

\usepackage[shortlabels]{enumitem}
\usepackage{aliascnt}
\usepackage[bookmarks=true,pdfstartview=FitH, pdfborder={0 0 0}, colorlinks=true,citecolor=red, linkcolor=blue]{hyperref}
\usepackage{bbm}
\usepackage{nicefrac}

\theoremstyle{plain}
\newtheorem{thm}{Theorem}[section]
\newaliascnt{cor}{thm}
\newaliascnt{prop}{thm}
\newaliascnt{lem}{thm}
\newtheorem{cor}[cor]{Corollary}
\newtheorem{prop}[prop]{Proposition}
\newtheorem{lem}[lem]{Lemma}
\aliascntresetthe{cor}
\aliascntresetthe{prop}
\aliascntresetthe{lem} 
\theoremstyle{definition}
\newaliascnt{defn}{thm}
\newaliascnt{asu}{thm}
\newaliascnt{con}{thm}
\aliascntresetthe{defn}
\aliascntresetthe{asu}
\aliascntresetthe{con}
\newcounter{stp}
\newcounter{stpi}
\newcounter{stpci}
\newcounter{stpiii}

\theoremstyle{thm}
\newaliascnt{rem}{thm}
\newaliascnt{exa}{thm}
\newaliascnt{masu}{thm}
\newaliascnt{nota}{thm}
\newaliascnt{sett}{thm}
\newtheorem{rem}[rem]{Remark}

\aliascntresetthe{rem}
\aliascntresetthe{exa}
\aliascntresetthe{masu}
\aliascntresetthe{nota}
\aliascntresetthe{sett}
\numberwithin{equation}{section}

\setlist[enumerate]{font = \normalfont}

\newcommand {\N}	{\mathrm{N}}
\newcommand {\Z}	{\mathbb{Z}}

\newcommand {\R}	{\mathbb{R}}
\newcommand {\C}	{\mathbb{C}}
\newcommand {\E}	{\mathbb{E}}

\newcommand {\T}	{\mathbb{T}}

\renewcommand{\d}{\, \mathrm{d}}

\DeclareMathOperator{\im}{Im}

\DeclareMathOperator{\divH}{div_{\H}}

\newcommand{\Hinfty}{\mathcal{H}^\infty}

\renewcommand{\Re}{\mathrm{Re}}
\renewcommand{\Im}{\mathrm{Im}}

\newcommand{\D}{\mathrm{D}}

\renewcommand{\H}{\mathrm{H}}

\newcommand{\per}{\mathrm{per}}

\newcommand{\sigmabar}{\bar{\sigma}}

\newcommand{\sssigmabar}[1]{{#1}_{\sigmabar}}

	\newcommand{\dk}[1]{\partial_{#1}}
	\newcommand{\dt}{\dk{t}} 
	\newcommand{\dz}{\dk{z}} 
	
	\renewcommand{\phi}{\varphi}
	
	\renewcommand{\bar}[1]{\overline{#1}}
	\newcommand{\vbar}{\bar{v}}
	\newcommand{\fbar}{\bar{f}}
	\newcommand{\vtilde}{\tilde{v}}

	\renewcommand{\div}{\mathrm{div} \, }
	\newcommand{\nablaH}{\nabla_{\H}}
	\newcommand{\DeltaH}{\Delta_{\H}}

    \newcommand{\Cof}{\mathrm{Cof}}
	
	\newcommand{\rC}{\mathrm{C}}
	\newcommand{\rL}{\mathrm{L}}
	\newcommand{\rW}{\mathrm{W}}
	\newcommand{\rH}{\H}
	\newcommand{\rB}{\mathrm{B}}

	\newcommand{\rLsigmabar}{\sssigmabar{\rL}}

	\newcommand{\rLq}{\rL^q}
	\newcommand{\rLp}{\rL^p}

	\newcommand{\rLqsigmabar}{\rLsigmabar^q}

	\newcommand{\rLqOmegasigmabar}{\rLqsigmabar(\Omega)}

	\newcommand{\rX}{\mathrm X}

\newcommand{\cA}{\mathcal{A}}
\title[Compressible primitive equations]{The Lagrangian approach to the compressible primitive equations}

\author{Matthias Hieber}
\address{Technische Universit\"at Darmstadt\\
	Fachbereich Mathematik\\
	Schlossgartenstr.\ 7\\
	64289 Darmstadt\\
	Germany}
\email{hieber@mathematik.tu-darmstadt.de}

\author{Yoshiki Iida}
\address{Department of Pure and Applied Mathematics \\
	Graduate School of Fundamental Science and Engineering \\
	Waseda University \\
	169-8555 Tokyo \\
	Japan}
\email{yoshiki-i737@asagi.waseda.jp }

\author{Arnab Roy}
\address{Basque Center for Applied Mathematics (BCAM), Alameda de Mazarredo 14, 48009 Bilbao, Spain.}
	\address{IKERBASQUE, Basque Foundation for Science, Plaza Euskadi 5, 48009 Bilbao, Bizkaia, Spain.}
\email{aroy@bcamath.org}

\author{Tarek Z\"{o}chling}
\address{Technische Universit\"at Darmstadt\\
	Fachbereich Mathematik\\
	Schlossgartenstr.\ 7\\
	64289 Darmstadt\\
	Germany}
\email{zoechling@mathematik.tu-darmstadt.de}

\begin{document}

\begin{abstract}
This article develops the hydrostatic Lagrangian approach to the compressible primitive equations. A fundamental aspect in the analysis is the investigation of the compressible hydrostatic Lam\'{e} and Stokes operators. Local strong well-posedness for large data and global strong well-posedness for small data are established under various assumptions on the pressure law, both in the presence and absence of gravity.
\end{abstract}
\maketitle
\section{Introduction}
\label{sec:intro}
The primitive equations of geophysical fluid dynamics are regarded as a standard model for oceanic and atmospheric dynamics and are derived from the 
Navier-Stokes equations by assuming a hydrostatic balance for the pressure term. They have been introduced in a series of papers  by Lions, Temam and Wang ~\cite{LTW:92a, LTW:92b, LTW:95}.

Very roughly speaking, the model developed by Lions, Temam and Wang couples the compressible primitive equations describing the dynamics of the  atmosphere and 
the incompressible primitive equations for the dynamics of the ocean through a  nonlinear  interface condition. In  the analysis of the 
compressible equation for the atmosphere they introduced a new coordinate system, the pressure coordinate system, in which, assuming the pressure of the atmosphere at the interface to be constant,
one essentially  obtains an incompressible primitive equation. This assumption is, however, not physical and hence  not true in realistic situations.      
In their pioneering work Lions, Temam and Wang proved the existence of a weak solution to the primitive equations and also to the coupled system; its uniqueness
remains an open problem for  both systems until today.

Since then the compressible and  incompressible primitive equations have been the subject of intensive  mathematical investigations. Indeed, a celebrated result of Cao ant Titi \cite{CT:07}
asserts that the incompressible primitive equations subject to Neumann boundary conditions are globally strongly well-posed in the three-dimensional setting  
for arbitrarily  large data belonging to $\rH^1$.  For related results concerning different approaches, different boundary conditions and also critical spaces, we refer to  \cite{KZ:07a,HK:16,GGHHK:20b}.  

The situation for the compressible primitive equations is much less understood. Given a cylindrical domain $\Omega =G \times (0,1)$, where $G =(0,1)^2$ the compressible primitive equations are 
given by the set of equations
\begin{equation}
	\left\{
	\begin{aligned}
        \dt \varrho +\div  (\varrho u)  &=0, &&\text{ on }\Omega \times (0,T), \\
		\varrho \left (\dt  v +u \cdot \nabla v \right )  - \mu \Delta v- \mu' \nablaH \divH v +  \nablaH p &= 0, &&\text{ on } \Omega\times (0,T),  \\
		  \dz p &= - g\varrho  , &&\text{ on } \Omega \times (0,T), \\ 
p &=  \varrho^\gamma, &&\text{ on } \Omega \times (0,T),  \\
v(0)&= v_0, \quad \varrho(0)=\varrho_0 \\
\end{aligned}
\right. 
\label{eq:compressibleprimtivegravity}
\end{equation}
for some $\gamma \geq 1$. Here $u = (v, w) : \Omega \rightarrow \R^3$ denotes the velocity, $p : \Omega \rightarrow \R$ the atmospheric pressure and $\varrho : \Omega \rightarrow \R$ the 
atmospheric density. Here we use $\nablaH$, $\divH$ for the horizontal gradient and the horizontal divergence, that is $\nablaH:=(\partial_x,\partial_y)^T$ and
$\divH:=\nablaH \cdot{}$. The Lam\'e coefficients $\mu$ and $\mu'$ are assumed to satisfy the standard conditions $\mu>0$ and  $\mu+ \mu'>0$.  

The above system  is supplemented with the boundary conditions
\begin{equation}\label{eq:bc}
\begin{aligned} 
    v|_{G \times \{1 \}} = 0, \quad (\partial_z v)|_{G \times \{0 \}} = 0, \quad w|_{G \times \{ 0\} \cup G \times \{ 1 \}} = 0 \ \text{ and } \ u,p,\varrho \ \text{ are periodic on }\ \partial G \times [0,1] .
\end{aligned} 
\end{equation}

The case $\gamma=1$ is of particular interest. Indeed, the law of an ideal gas tells us that $p=R \varrho T$, where $T$ denotes the temperature of the atmosphere and $R$ the gas constant.
This gives us in the isothermal situation  the relation  $p=c \varrho$ for some constant $c$. The articles  of Lions, Temam and Wang \cite{LTW:92a}, \cite{LTW:92b} also treat this
situation. For both classical and recent results on the compressible Navier-Stokes equations, we refer to the works of Lions \cite{PLL98} and Feireisl \cite{EF01} for weak solutions, and those of Matsumura and Nishida \cite{MN:83} and Enomoto and Shibata \cite{ES:18} for strong solutions. In particular, for the case of the exponent $\gamma=1$, we refer to the work of Hoff \cite{Hof:02} and Danchin and Mucha \cite{DM:23}.

The existence of global weak solutions for the compressible system with degenerate viscosities without gravity goes back to the work of Liu
and Titi \cite{LT:weak}. Ersoy, Ngom and Sy also constructed a global weak solution to some two dimensional version of compressible primitive equations in \cite{ENS:11} and \cite{EN:12}. 
Local strong well-posedness in  the three dimensional setting with $\gamma>1$ was established by Liu and Titi \cite{LT:21} in two cases: with gravity but without vacuum and with vacuum but without gravity.    
Furthermore, they showed in \cite{LT:20} a global, strong well-posedness result for the compressible primitive equations for data close to an incompressible flow in the situation without gravity.
Note that all these results are based on energy methods.

Inspired by the Lagrangian approach to the compressible Navier-Stokes equations, see e.g., \cite{Dan:14,Dan:18}, we introduce  here for the first time the {\em hydrostatic Lagrangian approach} to the
compressible primitive equations.
It differs from the usual Lagrangian approach since we define here the flow along the vertical average $\overline{v}$ of the horizontal velocity $v$, i.\ e., the flow $\rX$ is defined as the solution of the
equation 
\begin{equation}
\begin{aligned}
        \begin{cases}
            \dt \mathrm{X}(t,y_\H) &= \vbar (t,\mathrm{X}(t,y_\H)), \\
            \mathrm{X}(0,y_\H) &= y_\H,
        \end{cases}
    \end{aligned}
\end{equation}
where $y_\H$ denotes the horizontal coordinate.  

This hydrostatic Lagrangian approach allows us to prove first local, strong well-posedness of \eqref{eq:compressibleprimtivegravity} subject to \eqref{eq:bc} including gravity and general values of $\gamma$ for initial data $\varrho_0 \in \rH^{1,q}(\Omega)$ and $v_0 \in \rB^{2(1-1/p)}_{qp}(\Omega;\R^2)$ subject to suitable compatibility conditions and $p,q >2$.
Moreover, for the case $\gamma=1$ we obtain a global strong well-posedness result for small data even in the situation of gravity.

Some words about our approach are in order. The hydrostatic Lagrangian transformation leads us to a new set of equations, where the linearized part is given by the
{\em compressible,  hydrostatic Stokes operator} $A_{\mathrm{CHS}}$, see \autoref{sec:global}. We show that $-A_{\mathrm{CHS}}$ admits a bounded $\mathcal{H}^\infty$-calculus on $\rH_\per^{1,q}(G;\R)\times \rLq(\Omega;\R^2)$ of an $\mathcal{H}^\infty$-angle strictly
less than $\pi/2$. This allows us to perform a  maximal regularity approach to equation \eqref{eq:compressibleprimtivegravity} implying
our  local and global existence results even in the case of gravity. Here we use the fact that the compressible hydrostatic Stokes operator is invertible when restricted to the space of
mean value free functions and that the maximal regularity estimates can be shown in this case to be valid on the time interval $[0,\infty)$.    

The boundedness of the $\mathcal{H}^\infty$-calculus of $-A_{\mathrm{CHS}}$ on $\rH_\per^{1,q}(G;\R)\times \rLq(\Omega;\R^2)$ is based on perturbation theory of the {\em hydrostatic Lam\'e operator} $-A_{\mathrm{HL}}$ defined and
investigated  in \autoref{sec:lintheo}. Indeed, we prove in \autoref{sec:lintheo} that $-A_{\mathrm{HL}}$ is parameter elliptic on the cylindrical domain $\Omega$. Combining  the $\mathcal{H}^\infty$-result for parameter elliptic, cylindrical boundary value problems given in \cite{N:12} with the perturbation result given in \cite{HH:05}, we conclude that $-A_{\mathrm{HL}}$ admits a bounded $\mathcal{H}^\infty$-calculus on $\rLq(\Omega;\R^2)$ of angle strictly less than $\pi/2$.

This article is structured as follows. In \autoref{sec:prelim} we present after some preliminaries the  three main results of this article. The hydrostatic Lagrangian transform is introduced and investigated
then in \autoref{subsec:lagrangecor}. In \autoref{sec:lintheo} we introduce and study the hydrostatic Lam\'e operator which leads in \autoref{sec:localwp} to our local, strong well-posedness results. The proof of our
global strong well-posedness results for small data to  the compressible primitive equations is given in \autoref{sec:global} in various settings.

\section{Preliminaries and Main Results}
\label{sec:prelim}
Consider a cylindrical domain $\Omega =G \times (0,1)$, where $G =(0,1)^2$. For an integrable function $f \colon \Omega  \rightarrow \R^n$, with $n \in \{1,2\}$, we define the
vertical average of $f$ by $\fbar \coloneqq \int_0^1 f(\cdot, \cdot, \xi) \d \xi$.
\newpage 

\leftline{\em 1. The case $\gamma=1$ with gravity}

We start with the case $\gamma=1$ meaning that $p = c \varrho$ for some $c>0$. The pressure law is then given by the law of  ideal gas, which reads as  $p = RT\varrho$, where $T$ denotes temperature and
$R>0$ the gas constant. In the isothermal situation we obtain the relation $p=c \varrho$. We  consider the set of equations 
\begin{equation}
	\left\{
	\begin{aligned}
        \dt \varrho +\div  (\varrho u)  &=0, &&\text{ on }\Omega \times (0,T), \\
		\varrho \left (\dt  v +u \cdot \nabla v \right )  - \mu \Delta v- \mu' \nablaH \divH v +  \nablaH p &= 0, &&\text{ on } \Omega\times (0,T),  \\
		  \dz p &= - g\varrho  , &&\text{ on } \Omega \times (0,T), \\ 
p &=  c\varrho, &&\text{ on } \Omega \times (0,T),  \\
v(0)&= v_0, \quad \varrho(0)=\varrho_0 \\
\end{aligned}
\right. 
\label{eq:cpeprel}
\end{equation}
for some $c>0$. The system is supplemented by  the boundary conditions given in \eqref{eq:bc}.

For simplicity of the notation we set $c=g=1$ and  deduce then from \eqref{eq:cpeprel}$_3$ and \eqref{eq:cpeprel}$_4$ that 
\begin{equation}\label{eq:density}
    p(t,x,y,z) = \varrho(t,x,y,z) = \xi(t,x,y) \mathrm{e}^{-z},
\end{equation}
where $\xi$ denotes the pressure evaluated at the surface $G \times \{ 0\}$. In the following, we use the explicit vertical dependence of the density to obtain a two dimensional version of the
continuity equation by averaging in vertical direction. To this end, we consider the transformation 
\begin{equation}\label{eq:ztransform}
    z' = \frac{1-\mathrm{e}^{-z}}{\delta} \iff z = \log\big (\frac{1}{1-\delta z'}\big ) \quad \text{ for }\delta := 1- \mathrm{e}^{-1}>0,
\end{equation}
and introduce the new variables 
\begin{equation*}
    \widehat{v}(t,x,y,z') := v(t,x,y, z(z'))\quad \text{as well as} \quad \widehat{w}(t,x,y,z') := w(t,x,y, z(z')).
\end{equation*}
Note that the surface pressure $\xi$ remains unchanged and \eqref{eq:density} implies that $p = \varrho = \xi \cdot (1-\delta z')$.
Hence, naming the new variables  as before and averaging the continuity equation in vertical direction, we find the transformed system
\begin{equation}
	\left\{
	\begin{aligned}
       \dt \xi+ \vbar \cdot \nablaH\xi + \xi \divH  \vbar  &=0, &&\text{ on }\Omega \times (0,T), \\
       \xi \left ( \dt v + \bar{v} \cdot  \nablaH v \right ) -\frac{\mu\DeltaH v  }{1-\delta z} -\mu \dz \big ( \frac{1-\delta z}{\delta^2} \dz v \big ) -\frac{ \mu' \nablaH \divH v }
              {1-\delta z}&= F(\xi, v), &&\text{ on } \Omega \times (0,T),  \\
		  \dz \xi  &= 0  , &&\text{ on } \Omega \times (0,T), \\ 
\end{aligned}
	\right. 
	\label{eq:compressibleprimitivegravity2}
\end{equation}
where $F$ is given by
\begin{equation*}
     F(\xi,v) = - \xi \big((v-\bar{v}) \cdot \nablaH v + \frac{(1-\delta z)}{\delta}w \partial_{z}v \big) - \nablaH \xi.
\end{equation*}
Note that the transform \eqref{eq:ztransform} leaves the boundary conditions \eqref{eq:bc} unchanged. 

As for the incompressible primitive equations, where the vertical velocity  is determined by the divergence free  condition, see  e.\ g. \cite{HK:16}, we find that $w$ is determined by the
horizontal velocity $v$ and the density at the surface $\xi$. More precisely, 
comparing \eqref{eq:compressibleprimtivegravity}$_1$ and \eqref{eq:compressibleprimitivegravity2}$_1$ yields
\begin{equation}
    \label{eq:w}
    \xi w = - \int_0^z \divH (\xi \vtilde) \d z',
\end{equation} 
where $\vtilde = v -\vbar$. Finally, we prescribe the initial data as
\begin{equation*}
    (\xi(0), v(0)) = (\xi_0,v_0).
  \end{equation*}

\vspace{.2cm}  
\leftline{\em 2. The case $\gamma=2$ with gravity}  
We consider the system \eqref{eq:compressibleprimtivegravity} for the case  $p = \varrho^2$ and set  $g=1$.
In that case the relevant quantity for the density is still given by its value at the surface $\xi$. However, the representation of $\varrho$ by $\xi$ differs from the one given
in \eqref{eq:density}. For $\gamma=2$ we obtain
$\varrho = \xi +\frac{1}{2}z$ and after averaging the system reads as 
    \begin{equation}
	\left\{
	\begin{aligned}
        \dt \xi + \vbar \cdot \nablaH\xi + \xi \divH  \vbar+ \frac{1}{2}\overline{z \divH v} &=0, &&\text{ on }G \times (0,T), \\
		\big ( \xi + \frac{1}{2}z \big )\big (\dt  v+ u \cdot \nabla v \big ) -  \mu\Delta v - \mu'\nablaH \divH v  + \left ( 2 \xi + z \right )\nablaH \xi   &=0, &&\text{ on } \Omega \times (0,T),  \\
		  \dz \xi   &= 0  , &&\text{ on } \Omega \times (0,T). \\
	\end{aligned}
	\right. 
	\label{eq:compressiblegamma}
      \end{equation}
      The system is again supplemented by  the boundary conditions given in \eqref{eq:bc}. Here we can also recover the vertical velocity $w$ by
      comparing \eqref{eq:compressibleprimtivegravity}$_1$ and \eqref{eq:compressiblegamma}$_1$ which leads to
      \begin{equation*}
          \xi w = - \int_0^z \divH(\xi \vtilde) + \frac{1}{2} \widetilde{z \divH v} \d z'.
      \end{equation*}
Finally, we prescribe the initial data as 
\begin{equation*}
    (\xi(0), v(0)) = (\xi_0,v_0).
\end{equation*}

\vspace{.2cm}
\leftline{\em 3. The case of general pressure law  without  gravity}  
We finally  investigate the system \eqref{eq:compressibleprimtivegravity} without gravity but with a generalized pressure law of the form 
\begin{equation}\label{eq:pressurelaw}
    p = P(\varrho), \ \text{ where }P \in C^\infty \ \text{ with }\ c_1 \leq  P'(s) \leq c_2 \mbox{ for all } s \in \Im \ \varrho,
\end{equation}
where $c_1$ and $c_2$ are positive constants.
The corresponding system reads then as
\begin{equation}
	\left\{
	\begin{aligned}
        \dt \varrho +\div  (\varrho u)  &=0, &&\text{ on }\Omega \times (0,T), \\
		\varrho \left (\dt  v +u \cdot \nabla v \right )  - \mu \Delta v- \mu' \nablaH \divH v +  \nablaH p &= 0, &&\text{ on } \Omega\times (0,T),  \\
		  \dz p &= 0 , &&\text{ on } \Omega \times (0,T), \\ 
                      p &= P(\varrho), &&\text{ on } \Omega \times (0,T),
                \end{aligned}
	\right. 
	\label{eq:compressibleprimtivegeneralpressure}
\end{equation}
supplemented by the boundary conditions \eqref{eq:bc}.

We deduce from \eqref{eq:compressibleprimtivegeneralpressure}$_3$ as well as \eqref{eq:compressibleprimtivegeneralpressure}$_4$  that 
$P'(\varrho) \cdot \dz \varrho =0$. The general pressure law \eqref{eq:pressurelaw} then yields $\dz \varrho =0$. In particular, the density is independent of the vertical variable and
therefore it is sufficient to consider its value at the surface $\xi(t,x,y) \coloneqq \varrho(t,x,y,0) $. After averaging the continuity equation of
\eqref{eq:compressibleprimtivegeneralpressure}, taking into account the boundary conditions for $w$, the above system reads as 
\begin{equation}
	\left\{
	\begin{aligned}
        \dt \xi +\divH  (\xi \vbar)  &=0, &&\text{ on }G \times (0,T), \\
		\xi \left (\dt  v +u \cdot \nabla v \right )  - \mu \Delta v- \mu' \nablaH \divH v +  P'(\xi) \nablaH \xi &= 0, &&\text{ on } \Omega\times (0,T),  \\
		  \dz \xi &= 0 , &&\text{ on } \Omega \times (0,T). \\ 
	\end{aligned}
	\right. 
	\label{eq:compressibleprimtivegeneralpressure2}
\end{equation}
Here, the vertical velocity can be deduced as in \eqref{eq:w}. The system \eqref{eq:compressibleprimtivegeneralpressure2} is supplemented by the boundary conditions \eqref{eq:bc} and by the initial conditions 
\begin{equation*}
    (\xi(0), v(0))  = (\xi_0,v_0).
\end{equation*}

\begin{rem}\label{rem:systemeq}{\rm 
Let us note that the transformed systems \eqref{eq:compressibleprimitivegravity2} as well as \eqref{eq:compressiblegamma} are equivalent to the original system \eqref{eq:compressibleprimtivegravity},
where the  pressure and density are related  by the formula $p=\varrho^\gamma$ for $\gamma \in \{1,2\}$. In the case of $\gamma=1$, the transform \eqref{eq:ztransform} is invertible for all times and the
density can be recovered from \eqref{eq:density}. In the situation of $\gamma=2$, we can recover the density from $\varrho = \xi + \frac{1}{2}z$ and in the case of no gravity it even holds true that $\varrho = \xi$. In all cases regularity and positivity properties
transfer directly from $\xi$ to $\varrho$ as $z \in (0,1)$.

}
\end{rem}

At this point, some words about our approach for solving \eqref{eq:compressibleprimitivegravity2}, \eqref{eq:compressiblegamma} and \eqref{eq:compressibleprimtivegeneralpressure} are in order.
Let $s \in \R$ and $p,q \in (1,\infty)$. In the sequel, we
denote by $\rW^{s,q}(\Omega)$ the fractional Sobolev spaces, by $\rH^{s,q}(\Omega)$ the Bessel potential spaces and  by $\rB^{s}_{qp}(\Omega)$ the Besov spaces, where $\Omega \subset \R^3$ is an open.
As usual, we set $\rH^{0,q}(\Omega) := \rL^q(\Omega)$ and note that $\rH^{s,q}(\Omega)$ coincides with the 
classical Sobolev space $\rW^{m,q}(\Omega)$ provided  $s=m \in \N$. For more information on function spaces we refer e.g. to \cite{Ama:19} and \cite{Tri:78}. 

In the following, we need terminology to describe periodic boundary conditions on $\Gamma_l = \partial G \times [0,1]$ as well as on $\partial G$. Given $s \in [0,\infty)$ and $p,q \in (1,\infty)$ we define the spaces
\begin{equation*}
    \rH^{s,q}_\per (\Omega) := \bar{\rC^\infty_\per (\bar{\Omega})}^{\| \cdot \|_{\rH^{s,q}(\Omega)}} \ \text{ and } \ \rH^{s,q}_\per (G) := \bar{\rC^\infty_\per (\bar{G})}^{\| \cdot \|_{\rH^{s,q}(G)}},
\end{equation*}
where horizontal periodicity is modeled by the function spaces $\rC^\infty_\per(\bar{\Omega})$ and $\rC^\infty_\per(\bar{G})$ defined in \cite{HK:16}. Of course, we interpret $\rH^{0,q}_\per$ as $\rLq$. The Besov spaces $\rB^s_{qp,\per}(\Omega)$ and $\rB^s_{qp,\per}(G)$ are defined in a similar manner. For more information regarding spaces equipped with periodic boundary conditions, see also \cite{HK:16, GGHHK:20b, GGHHK:17}.
Concerning the solution and data spaces we first define the spaces
\begin{equation*}
\begin{aligned}
\rX_0 &:= \rH_\per^{1,q}(G;\R) \times \rLq(\Omega;\R^2) \quad \mbox{and} \\
\rX_1 &:= \rH_\per^{1,q}(G; \R) \times  \mathrm{Y}, \mbox{ where }  \mathrm{Y}:=  \{ v \in \rH_\per^{2,q} (\Omega;\R^2) \colon \ v = 0 \text{ on } \Gamma_u \ \text{ and } \ \dz v = 0 \text{ on }\Gamma_b  \},
\end{aligned}  
\end{equation*}
where we $\Gamma_u = G \times \{1\}$ and $\Gamma_b=G\times\{0\}$.
Then given $0< T \leq \infty$, we define the data space $\E_{0,T}$ and the solution space $\E_{1,T}$ by  
\begin{equation*}
    \begin{aligned}
        \E_{0,T} \coloneqq \rLp(0,T;\rX_0) 
        \quad \text{ and } \quad  
        \E_{1,T} \coloneqq \rH^{1,p}(0,T;\rX_0) \cap \rLp(0,T;\rX_1).
    \end{aligned}
\end{equation*}
Moreover, we introduce the solution and data spaces concerning only the velocity by
\begin{equation*}
    \begin{aligned}
        \E^v_{0,T} \coloneqq \rLp(0,T;\rLq(\Omega;\R^2)), \quad \E^v_{1,T } = \rH^{1,p}(0,T; \rLq(\Omega,\R^2)) \cap \rLp(0,T; \mathrm{Y}).
    \end{aligned}
\end{equation*}

We now in the position to formulate our  main theorems concerning the local well-posedness of \eqref{eq:compressibleprimitivegravity2} and \eqref{eq:compressiblegamma} as well as the
global well-posedness for small data of \eqref{eq:compressibleprimitivegravity2}, even  in the presence of gravity. Furthermore, in the absence of gravity, i.\ e., $g=0$, we also obtain local and global
well-posedness of \eqref{eq:compressibleprimtivegeneralpressure}, where the pressure and density are related  by a  monotone increasing function.

Our requirements on the set of initial data are collected in the following assumption.

\vspace{.2cm}
\noindent
{\bf Assumption (A)}: Let $p,q \in (2,\infty)$ such that $1 \neq \frac{1}{p} + \frac{1}{2q}$ and  $\frac{1}{2} \neq \frac{1}{p} + \frac{1}{2q}$ and assume that 
 \begin{enumerate}[(i)]
        \item $\xi_0 \in \rH_\per^{1,q}(G;\R)$ satisfies $M_1 \leq \xi_0 \leq M_2$ for  some constants $M_1, M_2 >0$  and all $(x,y) \in G$.  
        \item $v_0 \in \rB_{qp,\per}^{2(1-\nicefrac{1}{p})}(\Omega;\R^2) \ $ satisfies $v_0|_{\Gamma_u}=0 \ $ if $\ \frac{1}{2q}+ \frac{1}{p} < 1 \ $ and
          $\ (\dz v_0)|_{\Gamma_b }=0 \ $ if $\ \frac{1}{2}+\frac{1}{2q}+ \frac{1}{p} < 1$.
    \end{enumerate}

\begin{thm}[local strong well-posedness]\label{thm:localwpCAO} \mbox{}\\
Let $T>0$ and assume that $(\xi_0,v_0)$ satisfy assumption (A). 
Then there exists $0 < a =a(v_0) \leq T$ such that
    \begin{enumerate}[(a)]
        \item  the system  \eqref{eq:compressibleprimitivegravity2} admits a unique, strong solution $(\xi, v) \in \E_{1,a}$ satisfying
    \begin{equation*}
        \begin{aligned}
            \label{eq:solCAO}
            \xi \in \rH^{1,p}(0,a;\rH_\per^{1,q}(G;\R)), \quad
            v \in \rH^{1,p}(0,a; \rLq(\Omega,\R^2)) \cap \rLp(0,a; \mathrm{Y}) 
          \end{aligned}
    \end{equation*}
    and there exist constants $M_1^\ast$, $M_2^\ast>0$ such that $M_1^\ast \leq \xi \leq M_2^\ast$ for all $t \in (0,a)$. 
    \item The assertion $(a)$ holds true also for the system \eqref{eq:compressiblegamma}.
    \end{enumerate}
\end{thm}

Next, we state our  result concerning global well-posedness for small initial data.

\begin{thm}[global strong well-posedness for small data]\label{thm:globalWP} \mbox{} \\
Assume that $(\xi_0, v_0)$ satisfy assumption (A) and in addition  that
    \begin{equation*}
        \begin{aligned}
            \frac{1}{|G|}\int_G \xi_0 \d (x,y) = \bar{\xi} >0 \quad \text{ and } \quad
            \| (\xi_0-\bar{\xi},v_0) \|_{\rH^{1,q}(G) \times \rB_{qp}^{2(1-\nicefrac{1}{p})}(\Omega)}  \leq \delta
        \end{aligned}
      \end{equation*}
for some $\delta>0$.       
Then there exists $\eta_0 >0$ such that for all $\eta \in (0,\eta_0)$ there exists $\delta_0>0$ and $C>0$ such that for all $\delta \in (0,\delta_0)$  
the system \eqref{eq:compressibleprimitivegravity2} admits a unique, strong solution $(\xi,v) \in \E_{1,\infty}$ satisfying  
    \begin{equation*}
        \begin{aligned}
            &\| \xi-\bar{\xi} \|_{\rL^\infty(\R_+;\rH^{1,q}(G))} + \| \mathrm{e}^{\eta(\cdot)} \nablaH \xi \|_{\rH^{1,p}(\R_+; \rLq(G))} + \|  \mathrm{e}^{\eta(\cdot)} (\dt \xi) \|_{\rLp(\R_+; \rLq(G))} \\
            &+ \|  \mathrm{e}^{\eta(\cdot)} v \|_{\rLp(\R_+;\rH^{2,q}(\Omega))} + \|  \mathrm{e}^{\eta(\cdot)} (\dt v )\|_{\rLp(\R_+ ; \rLq(\Omega))} + \|  \mathrm{e}^{\eta(\cdot)} v \|_{\rL^\infty(\R_+ ; \rB_{qp}^{2(1-\nicefrac{1}{p})}(\Omega))}\leq C \delta. \\
        \end{aligned}
      \end{equation*}
Moreover, $\xi(t,x,y) \geq \nicefrac{\bar{\xi}}{2}$ for all $t \in \R_+$ and all $(x,y) \in G$.
\end{thm}

Finally, we obtain the following result dealing with the case of a general pressure law but without gravity. 

\begin{thm}[local strong and global strong well-posedness for small data]\label{thm:generalpressure}  \mbox{} \\
Assume that $(\xi_0, v_0)$ satisfy assumption (A).  
\begin{enumerate}[(a)]
\item Given  $T>0$ there exists $0 < a =a(v_0) \leq T$ such that the system
    \eqref{eq:compressibleprimtivegeneralpressure2}  admits a unique, strong solution $(\xi, v) \in \E_{1,a}$ satisfying $M_1^\ast \leq \xi \leq M_2^\ast$ for some
    constants $M_1^\ast$, $M_2^\ast>0$ and all $t \in (0,a)$. 
  \item There exists $\eta_0 >0$ such that for all $\eta \in (0,\eta_0)$ there exists $\delta_0>0$ and $C>0$ such that for all
    $\delta \in (0,\delta_0)$ the system \eqref{eq:compressibleprimtivegeneralpressure2} admits a unique, strong
    solution $(\xi,v) \in \E_{1,\infty}$ satisfying 
    \begin{equation*}
        \begin{aligned}
            &\| \xi-\bar{\xi} \|_{\rL^\infty(\R_+;\rH^{1,q}(G))} + \| \mathrm{e}^{\eta(\cdot)} \nablaH \xi \|_{\rH^{1,p}(\R_+; \rLq(G))} + \|  \mathrm{e}^{\eta(\cdot)} (\dt \xi) \|_{\rLp(\R_+; \rLq(G))} \\
            &+ \|  \mathrm{e}^{\eta(\cdot)} v \|_{\rLp(\R_+;\rH^{2,q}(\Omega))} + \|  \mathrm{e}^{\eta(\cdot)} (\dt v )\|_{\rLp(\R_+ ; \rLq(\Omega))} + \|  \mathrm{e}^{\eta(\cdot)} v \|_{\rL^\infty(\R_+ ; \rB_{qp}^{2(1-\nicefrac{1}{p})}(\Omega))}\leq C \delta. \\
        \end{aligned}
    \end{equation*}
Moreover,  $\xi(t,x,y) \geq \nicefrac{\bar{\xi}}{2}$ for all $t \in \R_+$ and all $(x,y) \in G$.
\end{enumerate}
\end{thm}

\begin{rem}{\rm 
Note  that the results given in the three theorems above are formulated  for the reformulated systems \eqref{eq:compressibleprimitivegravity2}, \eqref{eq:compressiblegamma} and
  \eqref{eq:compressibleprimtivegeneralpressure2}. In fact, the density evaluated at the surface $\xi$ is the relevant quantity for the describing the density $\varrho$. The above
  \autoref{rem:systemeq} shows that this is equivalent to solving the original systems \eqref{eq:compressibleprimtivegravity} and
  \eqref{eq:compressibleprimtivegeneralpressure}. Hence, the three well-posedness results stated above carry over to the original system  \eqref{eq:compressibleprimtivegravity} and
  \eqref{eq:compressibleprimtivegeneralpressure}
}
\end{rem}

\begin{rem}{\rm
By an approach similar to the one described for the  case $\gamma=2$, we verify  that the assertions of \autoref{thm:localwpCAO}$(a)$ hold true also for the general power law
  $p = \varrho^\gamma$, where $\gamma \geq 1$. However, for simplicity of the presentation we prove here in detail only the cases $\gamma \in \{1,2\}$.
    
}\end{rem}

\begin{rem}  \label{rem:energy}{ \rm
(a) The total mass of the system \eqref{eq:compressibleprimtivegravity} subject to $\gamma \geq 1$
  or \eqref{eq:compressibleprimtivegeneralpressure} is conserved, i. e.,  we have $\int_\Omega \rho = \int_\Omega \rho_0$.

  \noindent
(b) Consider the case with gravity and $\gamma=1$. Define the energy $E(t)$ by
        \begin{equation*}
            E(t) \coloneqq \int_\Omega \bigl ( \frac{1}{2}\xi |v|^2 + e(\xi) \bigr ) \ \text{ with }\ e(\xi) = \xi \log \xi +1 - \xi.
        \end{equation*}
        Taking inner products of \eqref{eq:compressibleprimitivegravity2}$_2$ with $v$ yields the energy balance
        \begin{equation*}
            E(t) + \int_0^t \bigl (\mu \| (1-\delta z)^{-\frac{1}{2}} \nablaH v \|^2_{2} + \frac{\mu}{\delta^2}  \| (1-\delta z)^{\frac{1}{2}} \dz v \|^2_{2} +\mu' \| (1-\delta z)^{-\frac{1}{2}} \divH v \|^2_{2} \bigr ) = E(0),
        \end{equation*}
        for all $t\in (0,a]$, where $a>0$ denotes the local existence time of the solution given in \autoref{thm:localwpCAO}.

 \noindent
(c) In the case without gravity and the general pressure law \eqref{eq:pressurelaw} we define the energy $\Tilde{E}(t)$ by
        \begin{equation*}
            \tilde{E}(t) \coloneqq \int_\Omega \bigl ( \frac{1}{2}\xi |v|^2 + \Tilde{e}(\xi) \bigr ) \ \text{ with }\ \Tilde{e}(\xi) =  \xi \int_1^\xi \frac{P(\xi)}{\xi^2} \d \xi - P(1)(\xi-1).
        \end{equation*}
        Then taking inner products of \eqref{eq:compressibleprimtivegeneralpressure2}$_2$ with $v$ yields the energy balance
        \begin{equation*}
            \tilde{E}(t) + \int_0^t \bigl (\mu \|  \nabla v \|^2_{2}  +\mu' \|  \divH v \|^2_{2} \bigr ) = \Tilde{E}(0),
        \end{equation*}
        for all $t\in (0,a]$, where $a>0$ denotes the local existence time given in \autoref{thm:generalpressure}
}
\end{rem}

\begin{rem}
    \label{rem:domain}{ \rm
      We may identify the domain $G=(0,1)^2$ subject to periodic boundary conditions with the two dimensional torus $\T^2$. Then the results described in the three theorem above hold true also for
  $\T^2 \times (0,1)$. Furthermore, \autoref{thm:localwpCAO} as well as \autoref{rem:energy} also hold true for $G = \R^2$. In that case the domain is given by a layer $\Omega_L := \R^2 \times (0,1)$.} 
\end{rem}

\section{Hydrostatic Lagrangian coordinates}\label{subsec:lagrangecor}

The purpose of this transformation is to mitigate the hyperbolic effects in the continuity equation. We term it the ``hydrostatic Lagrangian" as it defines the flow along 
$\vbar$, i.e., the solution 
$\mathrm{X}$ of the following differential equation
\begin{equation}
\label{eq:flow}
\left \{
    \begin{aligned}
            \dt \mathrm{X}(t,y_\H) &= \vbar (t,\mathrm{X}(t,y_\H)), \\
            \mathrm{X}(0,y_\H) &= y_\H,
    \end{aligned}
    \right. 
\end{equation}
where $y_\H = (y_1,y_2) \in G$. The flow $\mathrm{X}(t,\cdot) : G \rightarrow G$ is a well-defined $C^1$-diffeomorphisms provided that $\vbar$ is regular enough. Note that $\mathrm{X}$ is only a
two dimensional flow. Solving \eqref{eq:flow} leads to
\begin{equation}\label{eq:diffeo X Euler Lagrange sea ice para-hyper}
    \rX(t,y_\H) = y_\H + \int_0^t \ \bar{v}(s,\rX(s,y_\H)) \d s, 
\end{equation}
for $y_\H \in G$.
Denoting by $\mathrm{Y}(t,\cdot)$, the inverse of $\rX(t,\cdot)$, we find that
\begin{equation*}
    \nablaH \mathrm{Y}(t,\rX(t,y_\H)) = [\nablaH \rX]^{-1}(t,y_\H). 
\end{equation*}
We observe that the diffeomorphisms $\rX$ and $\mathrm{Y}$ depend on $\vbar$ and  more precisely on $v$.

For $\rX$ as in \eqref{eq:diffeo X Euler Lagrange sea ice para-hyper} and $p'$ denoting the H\"older conjugate of~$p$, we see due to 
 $\| \vbar \|_{\E^v_{1,T}} \leq C \| v\|_{\E^v_{1,T}}$ that
\begin{equation}\label{eq:est of nablaH X - Id sea ice para-hyper}
    \begin{aligned}
        \sup_{t\in [0,T]} \| \nablaH \rX - \mathbb{I}_2 \|_{\rH^{1,q}(G)}
        &\le C \int_0^T \| \nablaH \vbar(t,\cdot) \|_{\rH^{1,q}(G)}
        &\le C T^{\nicefrac{1}{p'}} \| v \|_{\E^v_{1,T}}
    \end{aligned}
\end{equation}
The embedding $\rW^{1,q}(G) \hookrightarrow \rL^\infty(G)$ yields
$\sup_{t\in [0,T]} \| \nablaH \rX - \mathbb{I}_2 \|_{\rL^\infty(G)} \le C T^{\nicefrac{1}{p'}} \| v \|_{\E^v_{1,T}}$. 
In particular, there exists $T'$ sufficiently small such that
\begin{equation}\label{eq:est of nablaHX - Id in Linfty}
    \sup_{t \in [0,T']} \| \nablaH \rX - \mathbb{I}_2 \|_{\rL^\infty(G)} \le \frac{1}{2}.
\end{equation}
A Neumann series argument thus guarantees the invertibility of $\nablaH \rX(t,\cdot)$ for all $t \in [0,T']$. Thus  $\nablaH \mathrm{Y}(t,\cdot)$ exists on $[0,T']$.
At the same time, we conclude by \eqref{eq:diffeo X Euler Lagrange sea ice para-hyper} that $\dt \nablaH \rX(t,y_\H) = \nablaH \vbar(t,y_\H)$.
Therefore, 
\begin{equation}\label{eq:timederX}
    \| \dt \nablaH \rX(t,\cdot) \|_{\rL^p(0,T;\rH^{1,q}(G))} \le C  \| v \|_{\E^v_{1,T}} .
\end{equation}
In summary, there exists a constant $C > 0$ such that
\begin{equation}\label{eq:est of nablaH X in W1p and Linfty}
    \| \nablaH \rX \|_{\rH^{1,p}(0,T;\rH^{1,q}(G))} + \| \nablaH \rX \|_{\rL^\infty(0,T;\rH^{1,q}(G))} \le C.
\end{equation}
Since $p$, $q>2$ the spaces~$\rH^{1,p}(0,T;\rH^{1,q}(G))$ and  $\rL^\infty(0,T;\rH^{1,q}(G))$ are Banach algebras.
The estimate  \eqref{eq:est of nablaH X in W1p and Linfty} yields  the existence of a constant $C > 0$ such that 
\begin{equation*}
    \begin{aligned}
        \| \det \nablaH \rX \|_{\rW^{1,p}(0,T;\rH^{1,q}(G))} + \| \det \nablaH \rX \|_{\rL^\infty(0,T;\rH^{1,q}(G))} 
        &\le C \text{ and }\\
        \| \Cof \nablaH \rX \|_{\rW^{1,p}(0,T;\rH^{1,q}(G))} + \| \Cof \nablaH \rX \|_{\rL^\infty(0,T;\rH^{1,q}(G))} 
        &\le C.
    \end{aligned}
\end{equation*}
Thanks to \eqref{eq:est of nablaHX - Id in Linfty} we find that $\det \nablaH \rX \ge C > 0$ on $(0,T) \times G$ for some constant $C > 0$ provided $T \le T'$.
The representation
\begin{equation}\label{eq:ex of nablaH Y sea ice para-hyper}
    \mathrm{Z} = [\nablaH \rX]^{-1} = \frac{1}{\det \nablaH \rX} (\Cof \nablaH \rX)^\top
\end{equation}
then results in
\begin{equation*}
    \|  \mathrm{Z} \|_{\rH^{1,p}(0,T;\rH^{1,q}(G))} + \| \mathrm{Z} \|_{\rL^\infty(0,T;\rH^{1,q}(G))}\ \le C \text{ and }\ \| \mathrm{Z} - \mathbb{I}_2\|_{\rL^\infty(0,T;\rH^{1,q}(G))} \leq CT^{\nicefrac{1}{p'}}.
\end{equation*}

We now consider the change of variables defined  by
\begin{equation*}
    \begin{aligned}
            V(t,y_\H,z) := v (t, \mathrm{X}(t,y_\H),z), \ W(t,y_\H,z) := w(t,\rX(t,y_\H),z),   \   \zeta(t,y_\H) := \xi(t, \mathrm{X}(t,y_\H)).
      \end{aligned}
\end{equation*}
Note that by this change of variables the initial values remain unchanged. The transformed equations in the case of $\gamma=1$ then read as
\begin{equation}
	\left\{
	\begin{aligned}
       \dt \zeta + \xi_0 \divH \bar{V} &=F_1(\zeta, V), &&\text{ on }G \times (0,T), \\
       \dt  V - \frac{\mu \DeltaH V }{(1-\delta z)\xi_0}- \frac{\mu}{\xi_0} \dz \big ( \frac{1-\delta z}{\delta^2} \dz V \big ) -
       \frac{\mu' \nablaH \divH V}{(1-\delta z)\xi_0}  &= F_2(\zeta, V), &&\text{ on } \Omega \times (0,T),  \\
		  \partial_{z} \zeta  &= 0  , &&\text{ on } \Omega\times (0,T), \\ 
        V&=0, &&\text{ on }\Gamma_u \times (0,T), \\
      \dz V &=0 , &&\text{ on } \Gamma_b \times (0,T), \\
        (\zeta(0), V(0)) &=(\xi_0, v_0),  
	\end{aligned}
	\right. 
	\label{eq:CAOatm}
\end{equation}
where the right-hand sides $F_1$ and  $F_2$ for $i=1,2$ are given by
\begin{equation}
    \begin{aligned}
        \label{eq:nonlinearities}
        F_1(\zeta,V)&= -(\zeta- \xi_0) \divH \bar{V} - \zeta \nablaH \bar{V} : \left [\mathrm{Z}^\top  - \mathbb{I}_2 \right ],\\
        (F_2(\zeta,V))_i&=\frac{\mu}{(1-\delta z)\xi_0}\Big (\sum_{l,j,k} \frac{\partial^2 V_i}{\partial y_l \partial y_k} (\mathrm{Z}_{k,j} - \delta_{k,j}) \mathrm{Z}_{l,j} +
        \sum_{l,k}\frac{\partial^2 V_i}{\partial y_l \partial y_k}\left (\mathrm{Z}_{l,k}- \delta_{k,l}\right) +  \sum_{l,j,k} \mathrm{Z}_{l,k} \frac{\partial V_i}{\partial y_k}
        \frac{\partial \mathrm{Z}_{k,j}}{\partial y_l} \Big) \\
        &\quad +\frac{\mu'}{(1-\delta z)\xi_0} \Big (\sum_{l,j,k} \frac{\partial^2 V_j}{\partial y_l \partial y_k}\left ( \mathrm{Z}_{k,j}- \delta_{k,j}\right)\mathrm{Z}_{l,i}  +
        \sum_{l,j}\frac{\partial^2 V_j}{\partial y_l \partial y_j} \left ( \mathrm{Z}_{l,i}- \delta_{k,i}\right) +
        \sum_{l,j,k} \mathrm{Z}_{l,i} \frac{\partial V_j}{\partial y_k} \frac{\partial \mathrm{Z}_{k,j}}{\partial y_l} \Big) \\
        & \quad +\big (1 -  \frac{\zeta}{\xi_0} \big )(\dt V )_i -\frac{\zeta}{\xi_0} \Big (\sum_{l,j}\tilde{V}_l \frac{\partial V_i}{\partial y_k} \mathrm{Z}_{l,j}
        -\frac{(1-\delta z)}{\delta} (W \dz V)_i \Big ) - \frac{ (\mathrm{Z}^\top \nablaH\zeta)_i}{\xi_0}. \\
        \end{aligned}
\end{equation}
Here we set  $\Tilde{V}:=V -\bar{V}$. The differences occurring in the above calculations for the case  $\gamma=2$ are being briefly discussed in the proof of \autoref{thm:localwpCAO}.

\section{Linear Theory: The Hydrostatic Lam\'e Operator}\label{sec:lintheo}

We start by  analyzing the linear equation 
\begin{equation}
	\left\{
	\begin{aligned}
        \dt \zeta + \xi_0 \divH \bar{V} &=G_1, &&\text{ on }G \times (0,T), \\
	\zeta(0, \cdot)  &= \xi_0  , &&\text{ on } G, \\ 
        \end{aligned}
	\right. 
	\label{eq:conteq3}
\end{equation}
for  the density $\zeta \colon G \rightarrow \R$ as well as the hydrostatic  Lam\'{e} equation
\begin{equation}\label{eq:Lame1}
	\left\{
	\begin{aligned}
          \dt  V - \frac{\mu \DeltaH V }{(1-\delta z)\xi_0}- \frac{\mu}{\xi_0} \dz \Big( \frac{1-\delta z}{\delta^2} \dz V \Big ) -  \frac{\mu' \nablaH \divH V}{(1-\delta z)\xi_0}
          +\omega V &=  G_2, &&\text{ on } \Omega \times (0,T),  \\
        V   &= 0  , &&\text{ on } \Gamma_u \times (0,T), \\ 
		 \dz V   &= 0  , &&\text{ on } \Gamma_b \times (0,T), \\ 
         V(0, \cdot)  &=v_0 ,  &&\text{ on } \Omega,
	\end{aligned}
	\right. 
\end{equation}
for  the velocity $V \colon \Omega \rightarrow \R^2$. Here $0 < T \le \infty$ is given and $(\zeta, V)$ are subject to periodic boundary conditions on the lateral boundary.

We associate with the left hand side of \eqref{eq:Lame1} the hydrostatic Lam\'e operator as follows: Let $\xi_0 \in \rH_\per^{1,q}(G;\R)$ and assume that $\xi_0 \geq c$ for some $c>0$.
Define functions $a:\Omega \to \R$ and $b:\Omega \to \R$ by
\begin{equation}\label{def:ab}
a(x,y,z):= \frac{1}{(1-\delta z)\xi_0}\ \mbox{ and }\ b(x,y,z):=\frac{(1-\delta z)}{\delta^2 \xi_0}
\end{equation}
and define the  differential operator $\cA(x,y,z,D)$ by  
$$
\cA(x,y,z,D) v:= \mu a(x,y,z)\DeltaH v + \partial_z(\mu b(x,y,z)\partial_zv) + \mu' a(x,y,z) \nablaH \divH v,
$$
subject to boundary conditions
\begin{equation}\label{Abc}
v = 0 \text{ on } \Gamma_u \  \text{ and }\  \dz v = 0 \text{ on }\Gamma_b.  
\end{equation}

For $q>1$ consider the $\rL^q$-realization of $\cA(x,y,z,D)$ in $\rLq(\Omega;\R^2)$ subject to \eqref{Abc} defined by
\begin{equation*}
\begin{aligned}
A_{\mathrm{HL}}v&:=  \cA(x,y,z,D) v, \\  
\D(A_{\mathrm{HL}})&:=  \{ v \in \rH_\per^{2,q} (\Omega;\R^2) \colon \ v = 0 \text{ on } \Gamma_u \ \text{ and } \ \dz v = 0 \text{ on }\Gamma_b  \} 
\end{aligned}
\end{equation*}
We call  $A_{\mathrm{HL}}$ the {\em hydrostatic Lam\'e operator}.

Our aim is to show that $-A_{\mathrm{HL}}+\omega$ admits a bounded $\mathcal{H}^\infty$-calculus on $\rLq(\Omega;\R^2)$ of $\mathcal{H}^\infty$-angle strictly less than $\pi/2$ provided
$\omega$ is large enough. In order to do so, we define the operator $\mathcal{A}_1(z,D)$ and its $\rLq$-realization given by
\begin{equation*}
   \mathcal{A}_1(z,D) := (1-\delta z) \xi_0 \cdot \mathcal{A}(x,y,z,D) \ \text{ and }\ A_1 v := \mathcal{A}_1(z,D) v, \ \text{ with }\ \D(A_1)=\D(A_{\mathrm{HL}} ).
\end{equation*}
We then treat the hydrostatic Lam\'{e} operator $-A_{\mathrm{HL}}$ as multiplicative perturbation of $A_1$ using the results in \cite{HH:05}. To this end, we show that the principal part $\cA_{1,\#}(z,\xi)$ of the symbol $\cA_{1}(z,\xi)$ is parameter elliptic in the cylindrical domain $\Omega$. The key property in the context of cylindrical parameter ellipticity is that the operator $\cA_1(z,D)$ is \emph{cylindrical}, meaning that it completely resolves into parts of which one acts only on $G=(0,1)^2$ and the other acts only on $(0,1)$. The operator $\cA_1(z,D)$ is represented as
\begin{equation*}
    \cA_1(z,D) v=\mu  \DeltaH v + \mu '\nablaH \divH v + \mu\frac{(1-\delta z)^2}{\delta^2}\partial_{zz}v -\mu \frac{(1-\delta z)}{\delta}\dz v= \cA_2(D_\H) v + \cA_3(z,D_z) v,
\end{equation*}
where we define  $\cA_2(D_\H)$ as well as $\cA_3(z,D_z)$ by
\begin{equation*}
    \cA_2(D_\H) v :=\mu  \DeltaH v + \mu '\nablaH \divH v \ \text{ and }\ \cA_3(z,D_z) v :=\mu\frac{(1-\delta z)^2}{\delta^2}\partial_{zz}v -\mu \frac{(1-\delta z)}{\delta}\dz v.
\end{equation*}
The $\rLq(G;\R^2)$-realization of $\cA_2(D_\H)$ 
is defined by
\begin{equation*}
    \begin{aligned}
        A_2v &:=\mathcal{A}_2 (D_\H)v, \\
        \D(A_2) &:= \{ \rH^{2,q}(G;\R^2) \colon \ v \text{ periodic on } \partial G \}
    \end{aligned}
\end{equation*}
and coincides with the classical Lam\'{e} operator on $(0,1)^2$ subject to periodic boundary conditions. The $\rLq(0,1;\R^2)$-realization of $\cA_3(z,D_z)$ is defined as
\begin{equation*}
    \begin{aligned}
        A_3 v&:=  \cA_3(z,D_z) v, \\
\D(A_3)&:=  \{ v \in \rH^{2,q} (0,1;\R^2) \colon \ v|_{z=1} = 0  \ \text{ and } \ \dz v|_{z=0}  = 0  \} .
    \end{aligned}
\end{equation*}
We now say that the principal part $-\cA_{1,\#}(z,\xi)$ is parameter elliptic of angle $\Phi_{-\cA_1} =\max \{ \Phi_{-\cA_2}, \Phi_{-\cA_3}  \}$ in the cylindrical domain $\Omega$ if and only if the principal part of the two dimensional Lam\'{e} operator $-\cA_{2,\#}(\xi_\H)$ is parameter elliptic of angle $\Phi_{-\cA_2}$ in the sense of \cite[Definition~7.4]{N:12} and the boundary value problem induced by $A_3$ is parameter elliptic of angle $\Phi_{-\cA_3}$ in the sense of \cite[Definition~8.2]{N:12}, meaning that the principal part $-\cA_{3,\#}(z,\xi_3)$ is parameter elliptic of angle $\Phi_{-\cA_3}$ for all $z\in (0,1)$ and the boundary value problem induced by $A_3$ satisfies the Lopatinski-Shapiro condition for all $\Phi > \Phi_{-\cA_3}$.

For more information on bounded
$\mathcal{H}^\infty$-calculus and parameter elliptic operators in the classical and in the cylindrical sense, we refer to \cite{Ama:95,DHP:03,DDHPV:04,N:12,DN:13,HNVW:16,GGHHK:17,GGHHK:20b,PS:16}.  

\begin{lem}\label{thm:elliptic}
 Let $q \in (1,\infty)$. The principal part $-\cA_{1,\#}(z,\xi)$ is parameter elliptic of angle $\Phi_{-\cA_1}=0$ in the cylindrical domain $\Omega = G \times (0,1)$.
 Moreover, there exists $\omega_0 \in \R$ such that for all $\omega > \omega_0$ the
  operator $-A_1+ \omega$ admits a bounded $\mathcal{H}^\infty$-calculus on $\rLq(\Omega;\R^2)$ with $\mathcal{H}^\infty$-angle $\Phi^\infty_{-A_1 +\omega} < \pi/2$.
 
\end{lem}

\begin{proof}
Consider the two dimensional Lam\'{e} operator $-A_2$ on $G=(0,1)^2$ subject to periodic boundary conditions. Its discrete symbol is given by
\begin{equation*}
    -\cA_{2,\#}(k_\H) = \begin{pmatrix}
        \mu |k_\H|^2 + \mu' k_1^2 & \mu' k_1 k_2 \\
         \mu' k_1 k_2 & \mu |k_\H|^2 + \mu' k_2^2
    \end{pmatrix},\ k_\H \in \Z^2.
\end{equation*}
Its eigenvalues $\lambda_{1/2}(k_\H)$ of $-\cA_{2,\#}(k_\H)$ are computed  to be 
    \begin{equation*}
        \begin{aligned}
            \lambda_{1/2}(k_\H) =\frac{2\mu + \mu'}{2} |k_\H|^2 \pm \frac{\mu'}{2} |k_\H|^2, \quad k_\H \in \Z^2.
        \end{aligned}
      \end{equation*}
      We see that $\lambda_{1/2}(k_\H)>0$ for all $k_\H \in \Z^2 \setminus \{ 0\}$ and it thus follows that $-\cA_{2,\#}(k_\H)$ is parameter elliptic of angle $\Phi_{-\cA_2}=0$ in the
      sense of \cite[Definition~7.4]{N:12}.
Next, consider the function $b_1 \colon \Omega \to \R$ defined by
\begin{equation*}
    b_1(x,y,z):=b_1(z):= \frac{(1-\delta z)^2}{\delta^2}.
\end{equation*}
Since $b_1$ is continuous, it satisfies the smoothness assumptions needed for applying \cite{DHP:03}. The principal part $-\cA_{3,\#}(z,\xi_3)$ is given by $-\cA_{3,\#}(z,\xi_3)= -b_1(z) \xi_3^2$
for all $z \in (0,1)$ and all $\xi_3 \in \R$. Hence, $-\cA_{3,\#}(z,\xi_3)$ is parameter elliptic of angle $\Phi_{-\cA_3}=0$ and even normally strongly elliptic for all $z\in (0,1)$. In view of the standard boundary conditions involved, it follows that the boundary value problem associated to $-A_3$ is normally elliptic and satisfies the conditions in \cite{DHP:03}. It thus
follows that $-\cA_{1,\#}(z,\xi)$ is parameter elliptic of angle $\Phi_{-\cA_1}=0$.

To establish the boundedness of the  $\mathcal{H}^\infty$-calculus of $-A_1$ up to a shift, we note that the lower order coefficient of $A_1$ only depends on the vertical variable and in
particular belongs to $\rL^\infty(\Omega)$. Moreover, the coefficients of the
corresponding boundary operator are even constant. 
 The assertion thus follows from the the H\o"lder continuity of  $b_1$ and the results in \cite[Section~8]{N:12} and \cite{DDHPV:04}.
\end{proof} 

We are now in position to apply \cite[Theorem~3.1]{HH:05} to establish the boundedness of the  $\mathcal{H}^\infty$-calculus of hydrostatic Lam\'{e} operator $-A_{\mathrm{HL}}$ up to a shift.

\begin{prop}\label{prop:Hunendlich}
Let $q \in (2,\infty)$ and assume that $\xi_0 \in \rH_\per^{1,q}(G;\R)$ satisfies $\xi_0 \geq c >0$ in $G$ for some $c>0$.  Then there exists $\omega_0 \in \R$ such that for all $\omega > \omega_0$ the
  operator $-A_{\mathrm{HL}}+ \omega$ admits a bounded $\mathcal{H}^\infty$-calculus on $\rLq(\Omega;\R^2)$ with $\mathcal{H}^\infty$-angle $\Phi^\infty_{-A_{\mathrm{HL}} +\omega} < \pi/2$.  
\end{prop}

\begin{proof}
A Sobolev embedding yields that  
  $((1-\delta z) \cdot \xi_0)^{-1} \in \mathrm{BUC}^\rho(\bar{\Omega})$ for some  $\rho \in (0,1)$. Define the multiplication operator $M$ on $\rLq(\Omega;\R^2)$ by 
  \begin{equation*}
      (M v)(x,y,z) := ((1-\delta z) \cdot \xi_0)^{-1} v(x,y,z), \ \text{ with }\ \D(M) :=\rLq(\Omega;\R^2) .
  \end{equation*}
  Then $M$ is boundedly invertible and admits a $\mathcal{H}^\infty$-calculus of angle $\Phi^\infty_M=0$ on $\rLq(\Omega;\R^2)$. We then write 
  \begin{equation*}
      \omega-A_{\mathrm{HL}} = \omega + M \cdot ( -A_1 +\omega_1) - M\cdot \omega_1,
  \end{equation*}
  where $\omega_1\geq 0$ denotes the shift obtained from \autoref{thm:elliptic} and the product is defined as in \cite{HH:05}. The commutator estimate 
  \begin{equation*}
      \| [M,(-A_1+\omega_1+\mu_1)^{-1}]( \mu_2+M)^{-1} \| \leq \frac{C}{(1+|\mu_2|)^{1-\alpha} |\mu_1|^{1+\beta}},\ \text{ for all }\ (\mu_1,\mu_2) \in \Sigma_{\Phi^\infty_{-A_1+\omega_1}} \times \Sigma_\pi,
  \end{equation*} 
  holds true whenever  $C,\alpha\geq 0$, $\beta >0$ and $\alpha+\beta<1$. Hence, there exists $\omega_0 \in \R$ such that the operator $\omega +M\cdot ( -A_1 +\omega_1)$ admits a
  bounded $\mathcal{H}^\infty$-calculus of angle strictly less than $\pi/2$ on $\rLq(\Omega;\R^2)$ for all $\omega >\omega_0$. The assentation then follows from  \cite[Theorem~3.1]{HH:05}. 
\end{proof}

The above \autoref{prop:Hunendlich} implies maximal $\rL^p-\rL^q$-regularity estimates  for the solution of \eqref{eq:Lame1}. We verify that the trace space
$\rX_\gamma^v = (\rX_0^v,\rX_1^v)_{1-1/p}$ for  $\rX_0^v=\rLq(\Omega;\R^2)$ and $\rX_1^v= \mathrm{Y}$ is characterized by  
\begin{equation*}
		\rX_\gamma^v \coloneqq \left \{
    \begin{aligned}
      \{ &v_0 \in\rB_{qp,\per}^{2(1-\nicefrac{1}{p})}(\Omega;\R^2) \colon (\dz v_0)|_{\Gamma_b} =0, \ v_0|_{\Gamma_u} = 0  \}, \quad &&\frac{1}{2} + \frac{1}{2q} + \frac{1}{p} < 1, \\
    	\{ & v_0 \in  \rB_{qp,\per}^{2(1-\nicefrac{1}{p})}(\Omega;\R^2) \colon  v_0|_{\Gamma_u} = 0 \}, \quad &&\frac{1}{2q}+ \frac{1}{p} < 1, \\
    	 & \rB_{qp,\per}^{2(1-\nicefrac{1}{p})}(\Omega;\R^2), \quad &&1 <  \frac{1}{2q}+ \frac{1}{p}.
    \end{aligned}
 	\right.
\end{equation*}
We then obtain the following maximal regularity estimate for the equation \eqref{eq:Lame1}.

\begin{cor}\label{lem:Lame equation}
Let $p,q \in (2,\infty)$, $G_2 \in \E^v_{0,\infty}$ and $v_0$ as in assumption (A). Then there exists $\omega \ge 0$ such that the equation \eqref{eq:Lame1} admits a unique, strong solution
 \begin{equation*}
    \begin{aligned}
        V &\in \E^v_{1,\infty} := \rH^{1,p}(\R_+;\rLq(\Omega;\R^2)) \cap \rLp(\R_+;\mathrm{Y}),
   \end{aligned}
\end{equation*}
satisfying 
\begin{equation*}
    \| V \|_{\E_{1,\infty}} \le C\bigl(\| G_2\|_{\E^v_{0,\infty}}  + \| v_0 \|_{ \rB_{qp}^{2(1-\nicefrac{1}{p})}(\Omega))}\bigr)
\end{equation*}
for some $C>0$. In particular, the solution operator 
\begin{equation} 
	S_V \colon \E^v_{0,\infty} \times  \rX_\gamma^v  \to \E^v_{1,\infty}, \quad (G_2,v_0) \mapsto V \label{eq:SV}
\end{equation}  
is an isomorphism. 
\end{cor}

Furthermore, we obtain the following well-posedness result for equation \eqref{eq:conteq3}.

\begin{cor}\label{cor:conteq}
Let $p,q \in (2,\infty)$ and  $V \in \E^v_{1,\infty}$ be the unique solution of \eqref{eq:Lame1}. Assume that  $G_1 \in \rLp(\R_+;\rH_\per^{1,q}(G;\R))$ and $\xi_0 \in \rH_\per^{1,q}(G;\R)$. Then \eqref{eq:conteq3}
admits  a unique, strong solution
\begin{equation*}
        \zeta \in \rH^{1,p}(\R_+; \rH_\per^{1,q}(G;\R))
    \end{equation*}
satisfying 
    \begin{equation*}
        \| \zeta \|_{\rH^{1,p}(\R_+; \rH^{1,q}(G,\R))} \le C \big ( \| G_1 \|_{\rLp(\R_+; \rH^{1,q}(G,\R))}+ \| V \|_{\E^v_{1,\infty}} + \| \xi_0 \|_{\rH^{1,q}(G)} \big )
    \end{equation*}
for some $C>0$. 
  \end{cor}

\section{Local-wellposedness}\label{sec:localwp}

The proof of \autoref{thm:localwpCAO} is based on the estimates given in \autoref{prop:Hunendlich} and on the estimates on the  nonlinearities $F_1$ and $F_2$
given in the sequel.

\begin{lem}\label{lem:nonlinearestimates}
Let $\tau<1$ and  $p,q \in (2,\infty)$. Then there exist $T \in (0,\tau]$, a constant $\delta >0$, depending only on $p$ and $q$,  as well as a constant $C>0$, depending only on $p$, $q$ and  $\tau$,
such that 
        \begin{equation*}
        \begin{aligned}
            \|F_1(\zeta,V) \|_{\rLp(0,T;\rH^{1,q}(G))} \leq T^{\delta} \| V \|_{\E^v_{1,T}} \quad \text{ and } \quad  
            \|F_2(\zeta,V) \|_{\E^v_{0,T}} \leq C\big (T^\delta+ T^\delta \| V \|_{\E^v_{1,T}}+ \| V \|^2_{\E^v_{1,T}}\big ),\\
        \end{aligned}
        \end{equation*}
whenever  $(\zeta, V) \in \E_{1,T}$.
\end{lem}

\begin{proof}
 We start by estimating  $\zeta \in \rH^{1,p}(0,T;\rH^{1,q}_\per(G;\R))$. By Sobolev embedding and  H\"older's inequality we obtain
 \begin{equation}
   \begin{aligned}\label{eq:zetaestimate}
        \| \zeta \|_{\rLp(0,T; \rH^{1,q}(G))} &\leq C  T^{\nicefrac{1}{p}}, \quad \| \zeta - \xi_0 \|_{\rL^\infty(0,T;\rH^{1,q}(G))}
        \leq C T^{\nicefrac{1}{p'}} \quad \text{and}  \quad  \| \zeta \|_{\rL^\infty(0,T;\rH^{1,q}(G))} \leq C.
        \end{aligned}
    \end{equation}
 In view of $\rH^{1,q}(G)$ being a Banach algebra for $q>2$ and estimate \eqref{eq:zetaestimate} we see that 
    \begin{equation*}
        \begin{aligned}
            \| F_1 \|_{\rLp(0,T;\rH^{1,q}(G))} &\leq  \| \zeta -\xi_0\|_{\rL^\infty (0,T;\rH^{1,q}(G))} \|\divH \bar{V} \|_{\rLp (0,T;\rH^{1,q}(G))} \\
            & \quad + \| \zeta \|_{\rL^\infty(0,T;\rH^{1,q}(G))}\| \mathrm{Z - \mathbb{I}_2} \|_{\rL^\infty(0,T;\rH^{1,q}(G))}\| \nablaH \bar{V}\|_{\rLp(0,T;\rH^{1,q}(G))} \\
           &\leq C T^{\nicefrac{1}{p'}} \| V \|_{\E^v_{1,T}}.
        \end{aligned}
    \end{equation*}
In order to estimate $F_2$ observe that by \cite[Section~6]{GGHHK:20b} we have $\|  V \cdot \nablaH V  \|_{\rLq(\Omega)} \leq C \| V \|^2_{\rH^{1+\nicefrac{1}{q},q}(\Omega)}$.
Hence, 
    \begin{equation*}
        \begin{aligned}
              \|\zeta \sum_{l,j}\Tilde{V}_l \frac{\partial V_i}{\partial y_k} \mathrm{Z}_{l,j}  \|_{\rLp(0,T;\rLq)} \leq C\|\zeta \|_{\rL^\infty(0,T;\rL^\infty(G))}\| \mathrm{Z} \|_{\rL^\infty (0,T;\rL^{\infty}(G))} \| V \|^2_{\rL^{2p}(0,T;\rH^{1+\nicefrac{1}{q},q}(\Omega))}  \leq C \| V \|^2_{\E^v_{1,T}}
        \end{aligned}
    \end{equation*}
by the embeddings
    \begin{equation}\label{eq:maxregembedding}
        \E^v_{1,T} \hookrightarrow \rL^{2p}(0,T; \rH^{2 -\nicefrac{1}{p},q}(\Omega)) \hookrightarrow \rL^{2p}(0,T;\rH^{1+\nicefrac{1}{q},q}(\Omega)),
      \end{equation}
for  $\nicefrac{1}{p}+ \nicefrac{1}{q} \le 1$.
Next, we estimate the nonlinear term including the vertical velocity, i.\ e., $W \cdot \dz V$, anisotropically and obtain 
    \begin{equation*}
      \| \frac{\zeta}{\xi_0}W \cdot  \dz V \|_{\rLp(0,T;\rLq(\Omega)} \leq C \|\zeta W \|_{\rL^\infty_t \rLq_\H \rL^\infty_z} \| V \|_{\rLp_t \rL^\infty_\H \rH^{1,q}_z}
      \leq C\left  \| \zeta W \right \|_{\rL^\infty_t \rLq_\H \rL^\infty_z} \| V \|_{\E^v_{1,T}}.
    \end{equation*}
Combining the embeddings $\rH^{1,q}_z \hookrightarrow \rL^\infty_z$ , $\E^v_{1, T} \hookrightarrow \rL^\infty(0, T; \rL^\infty_\H(G; \rLq_z(0, 1)))$ for $\nicefrac{1}{p}+\nicefrac{1}{q} < 1$  with 
the representation \eqref{eq:w}, we obtain 
    
    \begin{equation*}
        \begin{aligned}
             \left \| \zeta W \right \|_{\rL^\infty_t \rLq_\H \rL^\infty_z} &\leq C \|  \widetilde{V}\mathrm{Z}^\top \nabla_\H\zeta + \zeta \mathrm{Z}^\top : \nabla_\H \widetilde{V}\|_{\rL^\infty (0,T;\rLq(\Omega))} \\
             &\leq C \| \widetilde{V} \|_{\rL^\infty (0,T;\rL^\infty_\H \rLq_z)} \| \mathrm{Z} \|_{\rL^\infty (0,T;\rL^\infty(G))} \| \nablaH \zeta \|_{\rL^\infty (0,T;\rLq(G))} \\
             &\quad + C\| \zeta \|_{\rL^\infty (0,T;\rL^\infty(G))}  \| \mathrm{Z} \|_{\rL^\infty (0,T;\rL^\infty(G))} \| \nablaH \widetilde{V} \|_{\rL^\infty (0,T;\rLq(\Omega))} \\
             & \leq  C\| V \|_{\E^v_{1,T}}.
        \end{aligned}
    \end{equation*}

    Next, observe that 
    \begin{equation*}
      \big\| \big (1 -   \frac{\zeta}{\xi_0} \big )(\dt V ) \big\|_{\rLp(0, T; \rLq(\Omega))} \leq C\| \zeta - \xi_0 \|_{\rL^\infty(0,T;\rH^{1,q}(G))} \| V \|_{\rH^{1,p}(0,T;\rLq(\Omega))}
        \leq C  T^{\nicefrac{1}{p'}}\| V \|_{\E^v_{1,T}}.
    \end{equation*}
    We estimate the remaining terms by properties of the Lagrangian transformation given  in \autoref{subsec:lagrangecor} and  obtain
    \begin{equation*}
        \begin{aligned}
          \big \|  \sum_{l,j,k} \frac{\partial^2 V_i}{\partial y_l \partial y_k} (\mathrm{Z}_{k,j} - \delta_{k,j}) \mathrm{Z}_{l,j} \big \|_{\rLp(0,T;\rLq(\Omega))}
          &\leq C \| V \|_{\rLp(0,T;\rH^{2,q}(\Omega))} \| \mathrm{Z} - \mathbb{I}_2 \|_{\rL^\infty (0,T;\rH^{1,q}(G))} \| \mathrm{Z} \|_{\rL^\infty (0,T;\rH^{1,q}(G))} \\
          &\leq C T^{\nicefrac{1}{p'}}\| V \|_{\E^v_{1,T}}, \\
        \end{aligned}
     \end{equation*}   
      and similarly
\begin{equation*}
\begin{aligned}
          \big \|  \sum_{l,k}\frac{\partial^2 V_i}{\partial y_l \partial y_k}(\mathrm{Z}_{l,k}- \delta_{k,l}) \big \|_{\rLp(0,T;\rLq(\Omega))}
            &\leq C \| V \|_{\rLp(0,T;\rH^{2,q}(\Omega))} \| \mathrm{Z} - \mathbb{I}_2 \|_{\rL^\infty (0,T;\rH^{1,q}(G))} \leq C T^{\nicefrac{1}{p'}}\| V \|_{\E^v_{1,T}}.
        \end{aligned}
    \end{equation*}
    Furthermore,  observe that 
    \begin{equation*}
        \begin{aligned}
          \big \| \sum_{l,j,k} \mathrm{Z}_{l,k} \frac{\partial V_i}{\partial y_k} \frac{\partial \mathrm{Z}_{k,j}}{\partial y_l} \big \|_{\rLp(0,T;\rLq(\Omega))}
          &\leq C\| \mathrm{Z} \|_{\rL^\infty (0,T;\rL^{\infty}(G))} \| \nablaH V \|_{\rLp(0,T;\rH^{1,q}(\Omega))}  \sum_{l,j,k}
          \big \| \frac{\partial \mathrm{Z}_{k,j}}{\partial y_l} \big \|_{\rL^\infty (0,T;\rLq(G))} \\
            &\leq C T^{\nicefrac{1}{p'}}\| V \|_{\E^v_{1,T}},
        \end{aligned}
    \end{equation*}
    where we used that for $1 \leq j,k,l \leq 2$ and all $y_\H \in G$ one has  $\frac{\partial \mathrm{Z}_{k,j}}{\partial y_l}(0,y_\H) =0$. Hence, 
    \begin{equation*}
      \big \| \frac{\partial \mathrm{Z}_{k,j}}{\partial y_l} \big \|_{\rL^\infty (0,T;\rLq(G))} \leq T^{\nicefrac{1}{p'}}
      \big \|\frac{\partial \mathrm{Z}_{k,j}}{\partial y_l} \big \|_{\rH^{1,p}(0,T;\rLq(G))} \leq T^{\nicefrac{1}{p'}}  \|  \mathrm{Z} \|_{\rH^{1,p}(0,T;\rH^{1,q}(G))}  \leq CT^{\nicefrac{1}{p'}} .
    \end{equation*}
    Finally,  \eqref{eq:zetaestimate} allows us to estimate the pressure term as
    \begin{equation*}
     \big \|\frac{\zeta }{\xi_0 } \mathrm{Z}^\top \nablaH \zeta \big\|_{\rLp(0,T;\rLq(\Omega))} \leq C \| \zeta \|_{\rL^\infty(0,T;\rL^\infty(G))}\| \mathrm{Z} \|_{\rL^\infty (0,T;\rL^{\infty}(G))}
      \| \zeta\|_{\rLp(0,T;\rH^{1,q}(G))} \leq C T^{\nicefrac{1}{p'}}. 
    \end{equation*}
   
\end{proof}
    
We are now in position to prove our first local well-posedness result. 

\begin{proof}[Proof of \autoref{thm:localwpCAO}]
Let us remark first  that in the sequel we implicitly consider a shifted linearized equations in order to exploit \autoref{lem:Lame equation}. As the resulting linear term added on the right-hand side
  can be estimated easily, we do not explicitly mention this shift in the remainder of the proof.

  Observe next that \autoref{lem:Lame equation} guarantees  the existence of the reference solution $V_\ast$ to \eqref{eq:CAOatm} with  homogeneous right hand side  and with
  initial value $v_0$. 
In the following, the prescript $_0$ indicates homogeneous initial data. We then  set
\begin{equation*}
    \begin{aligned}
          \mathbb{B}_r\coloneqq \{(\zeta, V) \in\E_{1,T} \colon \ V-V_\ast \in \prescript{}{0}{\E^v_{1,T}}, \text{ and } \enspace \| V - V_\ast \|_{\E^v_{1,T}}  \le r \}.
    \end{aligned}
\end{equation*}
For $(\zeta_1, V_1 ) \in  \mathbb{B}_r$, we denote by
$$
(\zeta,V ) \coloneqq \Psi(\zeta_1, V_1) \in\E_{1,T}
$$
the solution to the linearized equation \eqref{eq:CAOatm} with right-hand sides $F_1(\zeta_1, V_1)$ and $F_2(\zeta_1, V_1)$ and guaranteed by  \autoref{lem:Lame equation}, \autoref{cor:conteq}
as well as by \autoref{lem:nonlinearestimates}. We show next that $\Psi$ is a self-map and a contraction.

Notice that $V - V_\ast \in  \prescript{}{0}{\E^v_{1,T}}$  and we hence only need to show that $ \| V - V_\ast \|_{\E^v_{1,T}} \leq r$. 
\autoref{lem:Lame equation} as well as \autoref{lem:nonlinearestimates} yield
\begin{equation*}
    \begin{aligned}
        \| V-V_\ast \|_{\E^v_{1,T}} &\leq C \| F_2(\zeta_1, V_1) \|_{\E^v_{0,T}} 
        \leq   C\big (T^\delta+T^\delta \| V_1 \|_{\E^v_{1,T}}+ \| V_1 \|^2_{\E^v_{1,T}} \big ) \\
        &\leq C \big ( T^\delta + T^\delta \| V_\ast \|_{\E^v_{1,T}}  + T^\delta \| V_1 - V_\ast \|_{\E^v_{1,T}} +  \| V_\ast \|^2_{\E^v_{1,T}}  +  \| V_1 - V_\ast \|^2_{\E^v_{1,T}} \big ) \\
        &\leq CT^\delta+ C T^\delta (\| V_\ast \|_{\E^v_{1,T}} +r) + C (\| V_\ast \|^2_{\E^v_{1,T}} + r^2),
    \end{aligned}
\end{equation*}
where the constant $C > 0$ is independent of $T$.  Observing that $ \| V_\ast \|_{\E^v_{0,T}} \rightarrow 0$ as $T \rightarrow 0$, the corresponding terms are getting arbitrary small by
choosing $T>0$ appropriately small. Hence, for appropriate $T$ and $r$, the solution map $\Psi$ is a self-map. The contraction property follows the same way by using the a priori estimate
on $\zeta$ given in  \autoref{cor:conteq} and the estimate on $F_1$ given in \autoref{lem:nonlinearestimates}.
 
Finally, we choose $T$ sufficiently small such that $\rX(t,\cdot)$ is a $C^1$-diffeomorphisms from $G$ into $G$ and set $\mathrm{Y}(t,\cdot) = \rX^{-1}(t,\cdot)$,
see \autoref{subsec:lagrangecor}. Then, defining 
\begin{equation*}
    \begin{aligned}
        \xi( t,x,y) := \zeta( t,\mathrm{Y} (t,x,y)) \quad \text{ and }\quad
        v (t,x,y,z) := V (t,\mathrm{Y}(t,x,y),z)
    \end{aligned}
\end{equation*}
we verify that $(\xi, v)\in \E_{1,T}$ satisfies the system \eqref{eq:compressibleprimitivegravity2}.
Since the vertical transform \eqref{eq:ztransform} is invertible for all times, we find a solution to \eqref{eq:compressibleprimtivegravity}.
The uniqueness property follows from the uniqueness of solutions to the system \eqref{eq:CAOatm}.
The  proof of assertion a) of \autoref{thm:localwpCAO} is complete. 

We finally  provide a proof of assertion  b). After introducing  Lagrangian coordinates for the system
\eqref{eq:compressiblegamma}, we see  that the hydrostatic Lam\'e operator in this case reads as  
\begin{equation*}
\begin{aligned}  
        B&: = \mu c(x,y,z)\Delta  + \mu' c(x,y,z) \nablaH \divH \\ 
       D(B)&:= \{ v \in \rH_\per^{2,q} (\Omega;\R^2) \colon \ v = 0 \text{ on } \Gamma_u \ \text{ and } \ \dz v = 0 \text{ on }\Gamma_b  \} 
\end{aligned}
 \end{equation*}
Here $c:\Omega \to \R$ is given by $c(x,y,z):=\frac{1}{\xi_0 + (1/2) z}$.
Arguing as in \autoref{sec:lintheo}, we see  that $-B+\omega$ admits a bounded $\mathcal{H}^\infty$-calculus on $\rLq(\Omega;\R^2)$ of $\mathcal{H}^\infty$-angle strictly less than $\pi/2$.  Moreover,  a result similar to
\autoref{cor:conteq} is true for the transformed equation \eqref{eq:compressiblegamma}$_1$,  which reads as 
    \begin{equation*}
      \dt \zeta + \zeta_0 \divH \bar{V} + \frac{1}{2}\int_0^1 z \divH V \d z = -(\zeta- \zeta_0) \divH \bar{V} - \zeta \nablaH \bar{V} :
      \big [\mathrm{Z}^\top  - \mathbb{I}_2 \big ] -\frac{1}{2} \int_0^1 z \nablaH V :\big [ \mathrm{Z}^\top -\mathbb{I}_2 \big ]\d z.
    \end{equation*}
    We then may conclude local well-posedness of equation \eqref{eq:compressiblegamma}. In fact,  the estimates of the nonlinear terms are  deduced similarly to the ones given
    in \autoref{lem:nonlinearestimates}. In particular, the transformed pressure term $( 2 \xi + z) (\mathrm{Z}^\top \nablaH \xi)$ may be estimated following the  arguments given  in the proof of 
\autoref{lem:nonlinearestimates}.
\end{proof}

\section{Global in time existence}\label{sec:global}
 
In order to prove the assertions of \autoref{thm:globalWP} we transform the system \eqref{eq:compressibleprimtivegravity}  by a slightly different  Lagrange transform compared to the one
introduced in \autoref{subsec:lagrangecor}. More precisely, we define the transformed density and velocity by
\begin{equation*}
    \zeta(t,y_\H) := \xi(t,\rX(t,y_\H)) - \bar{\xi} \quad \text{ and }\quad V(t,y_\H,z) := v(t,\rX(t,y_\H),z)),
\end{equation*}
where $\bar{\xi} = \frac{1}{|G|} \int_G \xi_0>0$.
The reason for this modification is our aim to  linearize the system \eqref{eq:compressibleprimitivegravity2} around the constant steady state $(\bar{\xi}, 0)$.
The respective system is then given by
\begin{equation}
	\left\{
	\begin{aligned}
        \dt \zeta + \bar{\xi} \divH \bar{V}  &=F_1(\zeta, V), &&\text{ on }G \times (0,\infty), \\
		\dt  V  - \frac{\mu \DeltaH V }{(1-\delta z)\bar{\xi}}- \frac{\mu}{\bar{\xi}} \dz \big ( \frac{1-\delta z}{\delta^2} \dz V \big ) -  \frac{\mu' \nablaH \divH V}{(1-\delta z)\bar{\xi}}  + \frac{\nablaH \zeta}{\bar{\xi}} &= F_2(\zeta, V), &&\text{ on } \Omega \times (0,\infty),  \\
		  \dz \zeta  &= 0  , &&\text{ on } \Omega \times (0,\infty), \\ 
        V &=0, &&\text{ on } \Gamma_u \times (0,\infty), \\
        \dz V &= 0, &&\text{ on } \Gamma_b \times (0,\infty), \\
   (\zeta(0), V(0)) &=(\xi_0-\bar{\xi}, v_0), 
	\end{aligned}
	\right. 
	\label{eq:CPEatmostransformed}
\end{equation}
where the right-hand sides $F_1$ and $F_2$ for $i=1,2$ are given by
\begin{equation}
    \begin{aligned}
        \label{eq:nonlinearities2}
        F_1(\zeta,V) &= -\zeta \divH \bar{V} - (\zeta + \bar{\xi}) \nablaH \bar{V} : \left [\mathrm{Z}^\top  - \mathbb{I}_2 \right ],\\
        (F_2(\zeta,V))_i &= \frac{\mu}{(1-\delta z)\bar{\xi}}\Big (\sum_{l,j,k} \frac{\partial^2 V_i}{\partial y_l \partial y_k} \left (\mathrm{Z}_{k,j}
          - \delta_{k,j}\right ) \mathrm{Z}_{l,j} + \sum_{l,k}\frac{\partial^2 V_i}{\partial y_l \partial y_k}\left (\mathrm{Z}_{l,k}- \delta_{k,l}\right)
        +  \sum_{l,j,k} \mathrm{Z}_{l,k} \frac{\partial V_i}{\partial y_k} \frac{\partial \mathrm{Z}_{k,j}}{\partial y_l} \Big) \\
        &\quad +\frac{\mu'}{(1-\delta z)\bar{\xi}} \Big (\sum_{l,j,k} \frac{\partial^2 V_j}{\partial y_l \partial y_k}\left ( (\mathrm{Z}_{k,j}- \delta_{k,j}\right)\mathrm{Z}_{l,i}
        +  \sum_{l,j}\frac{\partial^2 V_j}{\partial y_l \partial y_j} \left ( \mathrm{Z}_{l,i}- \delta_{k,i}\right)
        +  \sum_{l,j,k} \mathrm{Z}_{l,i} \frac{\partial V_j}{\partial y_k} \frac{\partial \mathrm{Z}_{k,j}}{\partial y_l} \Big) \\
        & \quad
        - \frac{\zeta (\dt V)_i}{\bar{\xi}}
        -\frac{\zeta + \bar{\xi}}{\bar{\xi}} \Big (\sum_{l,j}\Tilde{V}_l \frac{\partial V_i}{\partial y_k} \mathrm{Z}_{l,j}- (W \dz V)_i \Big )
        - \frac{1}{\bar{\xi}} \sum_j \frac{\partial \zeta}{\partial y_j}(\mathrm{Z}_{j, i} - \delta_{j,i} ).  \\
    \end{aligned}
\end{equation}

\subsection{The compressible hydrostatic Stokes operator}
\label{subsec:Lincom}\mbox{} \\
We  introduce now the compressible  hydrostatic  Stokes operator and discuss its properties. Recall that $\rX_0 \coloneqq  \rH_\per^{1,q}(G,\R) \times \rLq(\Omega,\R^2)$ and
$\rX_1\coloneqq  \rH^{1,q}_\per(G,\R) \times \mathrm{Y}$. For $\bar{\xi}$ given as above, define functions  $a_{\bar{\xi}}:\Omega \to \R$ and $b_{\bar{\xi}}:\Omega \to \R$ by
\begin{equation}\label{def:abar}
a_{\bar{\xi}}(x,y,z):= a_{\bar{\xi}}(z):=\frac{1}{(1-\delta z)\bar{\xi}} \mbox{ and } b_{\bar{\xi}}(x,y,z):=b_{\bar{\xi}}(z):=\frac{(1-\delta z)}{\delta^2 \bar{\xi}}.
\end{equation}
Furthermore, define the differential operator $\cA_{\bar{\xi}}$ by  
$$
\cA_{\bar{\xi}} v:= \mu a_{\bar{\xi}}(z)v + \partial_z(\mu b_{\bar{\xi}}(z)\partial_zv) + \mu' a_{\bar{\xi}}(z) \nablaH \divH v,
$$
subject to boundary conditions
\begin{equation}\label{Abcstokes}
v = 0 \text{ on } \Gamma_u  \text{ and }  \dz v = 0 \text{ on }\Gamma_b.  
\end{equation}

For $q >1$ consider the $\rL^q$-realization of $\cA_{\bar{\xi}}$ in $\rLq(\Omega;\R^2)$ subject to \eqref{Abcstokes} defined by
\begin{equation}\label{def:ahlxi}
\begin{aligned}
A_{\mathrm{HL},\bar{\xi}}v&:=  \cA_{\bar{\xi}} v \\
\D(A_{\mathrm{HL},\bar{\xi}})&:=  \{ v \in \rH_\per^{2,q} (\Omega;\R^2) \colon \ v = 0 \text{ on } \Gamma_u \ \text{ and } \ \dz v = 0 \text{ on }\Gamma_b  \} 
\end{aligned}
\end{equation}
We call  $A_{\mathrm{HL},\bar{\xi}}$ the {\em hydrostatic Lam\'e operator} subject to $\bar{\xi}$. Moreover, we define the \emph{compressible, hydrostatic Stokes operator} $A_{\mathrm{CHS}}$ in  $\rX_0$ by 
\begin{equation}
    \begin{aligned}\label{eq:matrix}
        A_{\mathrm{CHS}} &:= \begin{pmatrix}
             0 & -\bar{\xi} \divH \mathrm{avg} \\
             -\nablaH & A_{\mathrm{HL},\bar{\xi}}
           \end{pmatrix},\\
 \D(A_{\mathrm{CHS}}) &:=  \rX_1 = \rH_\per^{1,q}(G;\R) \times \D(A_{\mathrm{HL},\bar{\xi}}),  
\end{aligned}
\end{equation}
where $\mathrm{avg}(v) = \int_0^1 v(\cdot, z) \d z$.

For the investigation of the linear system \eqref{eq:CPEatmostransformed} we show first that the compressible hydrostatic Stokes operator $-A_{\mathrm{CHS}}$  admits a bounded  $\mathcal{H}^\infty$-calculus 
up to a shift.

\begin{prop}\label{prop:rsec}
  Let $q \in (1,\infty)$. Then there exists $\omega >0$ such that the operator $-A_{\mathrm{CHS}} + \omega$ admits a bounded $\mathcal{H}^\infty$-calculus on $\rX_0$ with
  $\mathcal{H}^\infty$-angle $\Phi^\infty_{-A_{\mathrm{CHS}}+ \omega} < \pi/2$. 
\end{prop}

\begin{proof}
Arguing similarly as in the proof of  \autoref{prop:Hunendlich}, we conclude that there exists a $\omega_0 \in \R$ such that for all $\omega>\omega_0$ the operator
  $-A_{\mathrm{HL},\bar{\xi}}+\omega$ admits a bounded $\mathcal{H}^\infty$-calculus on $\rLq(\Omega;\R^2)$ of angle $\Phi^\infty_{-A_{\mathrm{HL},\bar{\xi}}+ \omega} < \pi/2$.
We next decompose $A_{\mathrm{CHS}}$ into
\begin{equation}
    \begin{aligned}
     A_{\mathrm{CHS}} &= \begin{pmatrix}
             0 & -\bar{\xi} \divH \mathrm{avg} \\
             0 & A_{\mathrm{HL},\bar{\xi}}
           \end{pmatrix} +
           \begin{pmatrix}
             0 & 0 \\
             -\nablaH & 0
           \end{pmatrix} =: A_0 + A_1.
\end{aligned}
\end{equation}
In order to show the bounded $\mathcal{H}^\infty$-calculus of  $-A_{\mathrm{CHS}}$, note that $\bar{\xi}\divH \mathrm{avg}$ is a bounded operator from $\mathrm{Y}$ to $\rH_\per^{1,q}(G;\R)$. Hence, $A_0$ admits a bounded
$\mathcal{H}^\infty$-calculus on $\rH_\per^{1,q}(G;\R) \times \rLq(\Omega;\R^2)$.  Furthermore, since
$\nablaH$ maps $\rH_\per^{1,q}(G;\R^2)$ boundedly to $\rL^q(\Omega;\R^2)$, it follows that $A_{\mathrm{CHS}}$ is a bounded perturbation of $A_0$ and that thus $A_{\mathrm{CHS}}$ admits a bounded $\mathcal{H}^\infty$-calculus on $\rX_0$.
\end{proof}

We prove next that the compressible hydrostatic Stokes operator $\mathcal{A}$ is invertible when defined in a suitable subspace $\rX_m$ of $\rX_0$.
To this end,  we study the following resolvent problem within the $\rLq$-setting: given  $\lambda \in \C$ satisfying  $\Re \  \lambda \geq0$, consider the equation
\begin{equation}
	\left\{
	\begin{aligned}
        \lambda \zeta + \divH \bar{V}  &= f_1  , &&\text{ in } G, \\ 
        \lambda V  - \frac{\mu \DeltaH V }{(1-\delta z)}-\mu  \dz \big ( \frac{1-\delta z}{\delta^2} \dz V \big ) -
        \frac{\mu' \nablaH \divH V}{(1-\delta z)}  + \nablaH \zeta  &= f_2, &&\text{ on } \Omega ,  \\
        \dz \zeta  &= 0, &&\text{ on } \Omega, \\ 
        V &=0, &&\text{ on } \Gamma_u, \\
        \dz V &=0, &&\text{ on } \Gamma_b, \\
	\end{aligned}
	\right. 
	\label{eq:resolventhstokes}
\end{equation}
where we have set $\bar{\xi}=1$ for simplicity of the representation.

Integrating the first equation of \eqref{eq:resolventhstokes} and using the periodic boundary conditions we see that for the case $\lambda = 0$ we need to impose the compatibility condition 
$\int_G f_1 \d (x,y) =0$. 
Therefore, it is natural  to consider  the space 
\begin{equation*}
    \rL^q_0(G;\R) :=  \{ f \in \rL^q(G;\R) \colon \int_G f \d(x,y) =0 \}
\end{equation*}
and to restrict the operator $A_{\mathrm{CHS}}$ to the subspace $\rX_m \coloneqq \rH_\per^{1,q}(G;\R) \cap \rL^q_0(G;\R) \times \rL^q(\Omega;\R^2)$ of $\rX_0$ and equipped with the domain
$\D(A_{\mathrm{CHS}}) \cap \rX_m$.


\begin{lem}\label{lem:invert}
 Let $q \in [2, \infty)$ and $(f_1,f_2) \in \rX_m =\rH_\per^{1,q}(G;\R) \cap \rL^q_0(G;\R) \times \rL^q(\Omega,\R^2)$. Then there exists a unique solution $(\zeta,V) \in \D(A_{\mathrm{CHS}}) \cap \rX_m =\rH^{1,q}_\per(G;\R) \cap \rLq_0(G;\R) \times \mathrm{Y})$ of \eqref{eq:resolventhstokes} with $\lambda=0$ and a constant $C>0$ such that
\begin{equation}\label{eq:estimateinvert}
        \| \zeta \|_{ \rH^{1,q}(G)}+\| V\|_{\rH^{2,q}(\Omega)}  \leq C \left ( \| f_1 \|_{ \rH^{1,q}(G)}+ \| f_2\|_{\rLq(\Omega)} \right ).
\end{equation}
\end{lem}
  
\begin{proof}\label{proof:Leminvert}
The proof is inspired by  \cite[Chapter~3]{HK:16}. Note, however, that we need to consider here the hydrostatic Lam\'{e} type operator $A_{\mathrm{HL},\bar{\xi}}$  
instead of the Laplacian as in \cite[Chapter-3]{HK:16} and that  we have to deal with the  non-homogeneous divergence equation \eqref{eq:resolventhstokes}$_1$.
We subdivide the proof into two steps.

\noindent
{\em Step 1: the case $q=2$} \mbox{} \\
The existence of a unique solution $(\zeta,V)$ to \eqref{eq:resolventhstokes} combined  with the estimate \eqref{eq:estimateinvert} for the case $q=2$ follows from the Babu\v ska-Brezzi theory.
In fact, consider the weak formulation of the equation \eqref{eq:resolventhstokes}. For this we introduce the function spaces
        \begin{equation*}
            \mathcal{V}= \left\{ \varphi\in \rH_\per^{1,2}(\Omega;\R^2): \varphi=0\mbox{ on }\Gamma_u \right\}\quad\mbox{and}\quad \mathcal{W}= \rL^2_0(G;\R),
        \end{equation*} 
        and consider  the functions $a_{\bar{\xi}}$ and $b_{\bar{\xi}}$ defined as in \eqref{def:abar} for $\bar{\xi}=1$. 
    Suppose now that $(\zeta,v)$ is a strong  solution of \eqref{eq:resolventhstokes} for  $\lambda=0$.  Multiplying equation \eqref{eq:resolventhstokes}$_1$ and \eqref{eq:resolventhstokes}$_2$
    with  $(\psi,\varphi)\in \mathcal{W}\times\mathcal{V}$ and integrating by parts we verify  that the weak formulation of \eqref{eq:resolventhstokes} reads as  
\begin{equation}
	\left\{
	\begin{aligned}
          (\divH \bar{v},\psi)_{\rL^2(G)}  & = (f_1,\psi)_{\rL^2(G)},\\
          \mu \left (a_1 \nablaH v ,\nablaH \varphi\right )_{2} + \mu\left (b_1 \dz v, \dz \varphi\right )_{2} + \mu'\left (a_1 \divH v , \divH {\varphi} \right )_{2}
          - (\zeta, \divH \bar{\varphi} )_{\rL^2(G)}  &= (f_2,\varphi)_{2}.
        \end{aligned}
	\right. 
	\label{eq:weakstokesG}
\end{equation}
We then define the sesquilinear forms $\mathcal{C}: \mathcal{V}\times \mathcal{V}\rightarrow\mathbb{C}$ and  $\mathcal{D}:\mathcal{V}\times \mathcal{W}\rightarrow\mathbb{C}$ by 
\begin{equation*}
 \begin{aligned}
   \mathcal{C}(v,\varphi)&:=\mu \left (a_1 \nablaH v ,\nablaH \varphi\right )_{\rL^2(\Omega)} + \mu\left (b_1 \dz v, \dz \varphi\right )_{\rL^2(\Omega)} +
   \mu'\left (a_1 \divH v , \divH {\varphi} \right )_{\rL^2(\Omega)}, \\
 \mathcal{D}(v,\psi)&:= -(\psi, \divH \bar{v})_{\rL^2(G)},
    \end{aligned}
\end{equation*}
and set
\begin{equation*}
     \langle \chi_1,\psi\rangle:= (-f_1,\psi)_{\rL^2(G)},\quad \langle \chi_2,\psi\rangle:= (f_2,\varphi)_{\rL^2(\Omega)}.
\end{equation*}
Then  $\mathcal{C}$ and  $\mathcal{D}$ are bounded sesquilinear forms, since $a_1$ and $b_1$ are uniformly bounded and non-degenerate. Observe that $\chi_1\in \mathcal{W}^{*}$ and
$\chi_2\in \mathcal{V}^{*}$. This allows us to rephrase equation \eqref{eq:weakstokesG} as 
\begin{equation}
	\left\{
	\begin{aligned}
          \mathcal{D}(v,\psi)  & = \langle \chi_1,\psi\rangle ,\quad  \psi\in \mathcal{W},\\
          \mathcal{C}(v,\varphi) + \mathcal{D}(\varphi,\zeta)  &= \langle \chi_2,\psi\rangle,\quad  \varphi\in \mathcal{V}.
        \end{aligned}
	\right. 
	\label{eq:BBstokesG}
\end{equation}
 We use the identity: $\operatorname{curl}_\H^{\top}(\operatorname{curl}_\H v )= \nablaH \operatorname{div}_\H v - \Delta_\H v$ to obtain
        \begin{equation*}
            \begin{aligned}
                 \mathcal{C}(v,v) = \mu a_1 \| \operatorname{curl}_\H v \|^2_{\rL^2(\Omega)} + \mu b_1 C\| \partial_z v \|^2_{\rL^2(\Omega)} + (\mu+\mu') \| \divH v \|^2_{\rL^2(\Omega)}.
            \end{aligned}
        \end{equation*}
The  coerciveness of $\mathcal{C}$ will follow from the fact that $\mu>0$, $\mu+\mu' >0$ and the coefficients $a_1$, $b_1$ are
uniformly bounded from below. The inf-sup condition for $\mathcal{D}$ follows from \cite[Lemma 3.2]{HK:16}. We apply the Babu\v ska-Brezzi theorem to conclude that there exists a unique solution $(\zeta,v)\in \mathcal{W}\times \mathcal{V}$ to
\eqref{eq:BBstokesG}. Employing the method of difference quotients we  obtain higher regularity estimates and thus  $(\zeta,v) \in \rH_\per^{1,2}(G,\R) \cap \rL^{2}_0(G;\R) \times \mathrm{Y}$.

\vspace{.1cm}
\noindent
{\em Step 2: the case $q>2$} \mbox{}\\
We reduce the problem to a two-dimensional Stokes equation and an elliptic boundary value problem.  Multiplying \eqref{eq:resolventhstokes}$_2$ with  $(1-\delta z)$ and taking vertical averages yields
    \begin{equation}
	\left\{
	\begin{aligned}
          - \mu\Delta_\H \bar{v}  +\nablaH \Tilde{\zeta}  &= \bar{(1-\delta z)f_2} + \mu \frac{(1-\delta)^2}{\delta^2} (\dz v)|_{ \Gamma_u} +\mu \frac{1}{\delta} v|_{\Gamma_b}
          +\mu \bar{v}  +\mu' \nablaH f_1, &&\text{ in } G ,  \\
          \divH \bar{v} &= f_1, &&\text{ in }G,\\
        \end{aligned}
	\right. 
	\label{eq:stokesG}
\end{equation}
where $\Tilde{\zeta} = \left (1- \frac{\delta}{2}\right )\zeta $. If $\zeta$ is known, the velocity $v$ may then be determined by the elliptic problem

\begin{equation}
\left\{
\begin{aligned}
- a_1\mu \DeltaH v -\mu  \dz \left ( b_1 \dz v \right ) - a_1 \mu' \nablaH \divH v  &= f_2 - \nablaH \zeta, &&\text{ on } \Omega ,  \\
v &= 0, &&\text{ on } \Gamma_u  \\
\partial_z v &= 0, && \text{ on } \Gamma_b. 
\end{aligned}
\right. 
\label{eq:inhomLame}
 \end{equation}
Assume now that $2<q \leq4$. Then $f_1 \in \rH_\per^{1,2}(G;\R) \cap \rL^2_0(G;\R)$, $f_2 \in \rL^2(\Omega;\R^2)$ and there exists a unique solution $(\zeta,v) \in \rH_\per^{1,2}(G;\R) \cap \rL^2_0(G;\R) \times \mathrm{Y}$
 to \eqref{eq:resolventhstokes}. Note  that \eqref{eq:stokesG} combined with  \eqref{eq:inhomLame} are equivalent to \eqref{eq:resolventhstokes}. Hence, equations \eqref{eq:stokesG}
 and \eqref{eq:inhomLame} have a unique, strong  solution as well. By Sobolev embeddings and the trace theorem, each term on the right-hand side of \eqref{eq:stokesG}$_1$
 belongs to $\rL^q(G;\R^2)$ and we also see that the right-hand side $f_1$ of \eqref{eq:stokesG}$_2$ lies in $\rH_\per^{1,q}(G;\R)$. Applying \cite[Lemma~3.5]{HK:16}, we deduce that
 $\Tilde{\zeta} \in \rH_\per^{1,q}(G;\R) \cap \rLq_0(G;\R)$, which implies the same regularity for $\zeta$. Thus each term on the right-hand side of \eqref{eq:inhomLame} lies in $\rLq(\Omega;\R^2)$.
 Hence there exists  a unique, strong  solution $v \in Y$ to \eqref{eq:inhomLame} by the results in \cite{DHP:03} and thus $(\zeta,v) \in  \D(A_{\mathrm{CHS}}) \cap \rX_m$.
 We now may repeat this  argument for $4<q<\infty$ to conclude that the solution constructed in the $\rL^2-$framework
 admits $\rH^{1,q} - \rH^{2,q}-$regularity.
\end{proof}

The following  lemma concerns the solvability of the resolvent problem \eqref{eq:resolventhstokes} for $\lambda \in \C\backslash\{0\}$  satisfying $\Re \  \lambda  \geq 0$.

\begin{lem}
    \label{lem:reshstokes}
    Let $\lambda \in \C \setminus \{ 0\}$ with $\Re \  \lambda \geq0$, $q \in [2, \infty)$,  $(f_1,f_2) \in \rX_m $. Then there exists a unique solution
    $(\zeta,V) \in \D(A_{\mathrm{CHS}}) \cap \rX_m$ of \eqref{eq:resolventhstokes} satisfying
    \begin{equation}\label{eq:estimateinvert2}
        \| \zeta \|_{ \rH^{1,q}(G)}+| \lambda|\|V \|_{\rL^q(\Omega)}+\| V\|_{\rH^{2,q}(\Omega)}  \leq C \left ( \| f_1 \|_{ \rH^{1,q}(G)}+ \| f_2\|_{\rLq(\Omega)} \right ).
    \end{equation}
for some $C>0$. 
  \end{lem}
\begin{proof}
    As before we first the case $q=2$ and use then a bootstrap argument to obtain the case $q>2$. \\
   {\em Step 1: the case $q=2$} \mbox{} \\ Let us fix $\lambda\in \C \setminus \{ 0\}$ with $\Re \  \lambda \geq0$.
        Using \eqref{eq:resolventhstokes}$_1$ we find $\zeta = \frac{1}{\lambda}(f_1 - \divH \bar{v})$ and hence the equation \eqref{eq:resolventhstokes} can be rewritten as
        \begin{equation}\label{eq:ressimple}
            \lambda v - a_1 \mu \Delta v - \mu \dz ( b_1 \dz v)-a_1 \mu' \nablaH \divH v - \frac{1}{\lambda} \nablaH \divH \bar{v} = f_1 - \frac{1}{\lambda} \nablaH f_2,  \text{ in }\Omega,
        \end{equation}
        where $a_1$ and $b_1$ are given as in the proof of \autoref{lem:invert}.
        Consider the weak formulation of the above problem given by
        \begin{equation}\label{eq:restweak}
            \begin{aligned}
                \lambda(v,\varphi)_{2}+ \mu \left (a_1 \nablaH v ,\nablaH \varphi\right )_{2} + \mu\left (b_1 \dz v, \dz \varphi\right )_{2} + \mu'\left (a_1 \divH v , \divH {\varphi} \right )_{2}\\
                + \frac{1}{\lambda}(\divH \bar{v}, \divH \bar{\varphi})_{\rL^2(G)} 
                = \bigl (f_2- \frac{1}{\lambda} \nablaH f_1, \varphi \bigr )_{2}
            \end{aligned}
        \end{equation}
        for $\varphi \in \mathcal{V}$. We rephrase the equation \eqref{eq:restweak} as 
        \begin{equation*}
            \mathcal{C}_\lambda (v,\varphi) = (f_\lambda,\varphi)_{\rL^2(\Omega)}, \quad \text{ for all }\ \varphi \in \mathcal{V},
        \end{equation*}
        and note that by H\"older's inequality it is clear that $\mathcal{C}_\lambda :\mathcal{V} \times \mathcal{V} \rightarrow \C $ is a bounded sesquilinear form. Since $a_1$ and $b_1$ are uniformly bounded from below, we conclude that  $\mathcal{C}_\lambda$ is coercive as in the proof of Lemma \ref{lem:invert}. Now we apply Lax-Milgram theorem to conclude that there exists a unique solution $v \in \mathcal{V}$. We employ the method of difference quotients to obtain higher regularity estimates and to obtain $v \in \rH_\per^2(\Omega;\R^2)$ as well as $\zeta \in \rH_\per^{1}(G;\R) \cap \rL^2_0(G\R)$ from \eqref{eq:resolventhstokes}$_1$.
\vspace{.1cm}        
\noindent \\
{\em Step 2: the case $q>2$} \mbox{} \\
Averaging \eqref{eq:ressimple} we obtain the following  inhomogeneous two-dimensional Lam\'{e} type equation in $G$ 
\begin{equation}\label{eq:lamevbar}
  - \mu  \Delta_\H  \bar{v}  - \left ( \mu' + \lambda' \right ) \nablaH \divH\bar{v}  = \bar{(1- \delta z )f_2} - \lambda' \nablaH f_2
  + \mu \frac{(1-\delta)^2}{\delta^2} (\dz v)|_{ \Gamma_u} +\mu \frac{1}{\delta} v|_{\Gamma_b} +\mu \bar{v} - \lambda \bar{(1-\delta z) v},
\end{equation}
where $\lambda' = \frac{1}{\lambda} \left ( 1 -\frac{\delta}{2}\right )>0$. Note that if $\bar{v}$ is known we may find  $V$ by solving the equation 
      \begin{equation}
	\left\{
	\begin{aligned}
        \lambda v - a_1\mu \DeltaH v -\mu  \dz \left ( b_1 \dz v \right ) - a_1 \mu' \nablaH \divH v  &= f_2 + \frac{1}{\lambda} \nablaH \divH \bar{v}, &&\text{ on } \Omega ,  \\
          v &=0, &&\text{ on } \Gamma_u, \\
        \dz v &= 0, &&\text{ on } \Gamma_bv. 
\end{aligned}
\right.
\label{eq:inhomLame1}
 \end{equation}

Note that solving \eqref{eq:resolventhstokes} is equivalent to solving the two-dimensional problem \eqref{eq:lamevbar} as well as the elliptic boundary value problem \eqref{eq:inhomLame1}.
Now a bootstrap argument as in the proof of \autoref{lem:invert} can be used to deduce the  existence of  a unique solution $(\zeta,v)$ to  \eqref{eq:resolventhstokes} within the $\rLq$-framework. 
\end{proof}

We are finally  in position to prove that the part of $A_{\mathrm{CHS}}$ on $\rX_m \coloneqq \rH_\per^{1,q}(G;\R) \cap \rL^q_0(G;\R) \times \rL^q(\Omega;\R^2)$ generates an exponentially stable semigroup.

\begin{prop}\label{cor:expstable}
Let $q \in [2,\infty)$. The compressible hydrostatic Stokes operator $A_{\mathrm{CHS}}$ defined on $\rX_m$ generates an exponentially stable semigroup  $(\mathrm{e}^{tA_{\mathrm{CHS}}})_{t\geq 0}$ on $\rX_m$.
.
\end{prop}

\begin{proof}
Combining \autoref{lem:invert} and \autoref{lem:reshstokes} we see  that $\{ \lambda \in \C \colon \Re\ \lambda \geq 0 \} \subset \rho(A_{\mathrm{CHS}})$. Moreover,
\autoref{prop:Hunendlich} implies that there exist constant $C>0$, $\omega \geq0$ and $\varphi \in [0,\pi/2)$ such that 
   \begin{equation*}
       \big \| (\lambda - A_{\mathrm{CHS}})^{-1} \big \|_{\mathscr{L}(\rX_m)} \leq C,
     \end{equation*}
for all $\lambda \in \omega + \sum_{\pi-\varphi}$. Moreover, the set $\{\lambda \in \C \colon \Re\ \lambda \geq 0 \} \setminus \{ \omega + \textstyle \sum_{\pi-\varphi}\}$
is compact and  we deduce the existence of a constant $C>0$ such that
   \begin{equation*}
         \big \|\left (\lambda - A_{\mathrm{CHS}}\right)^{-1} \big \|_{\mathscr{L}(\rX_m)} \leq C \quad  \mbox{ for all }\lambda \in \C \text{ with }\Re\ \lambda \geq 0.
   \end{equation*}
We thus  conclude that the spectral bound $s(A_{\mathrm{CHS}})$ of $A_{\mathrm{CHS}}$ satisfies $s(A_{\mathrm{CHS}}) = -\eta_0$ for some $\eta_0>0$. Since the spectral bound of generators of analytic semigroups coincide 
with the growth bound of the associated semigroups, see e.\ g., \cite[Theorem 5.1.12]{ABHN:11}, the proof is complete.      
\end{proof}

The fact that the spectral bound of $A_{\mathrm{CHS}}$ is strictly negative allows us to  sharpen the assertion of \autoref{prop:rsec} and to  deduce by \cite[Proposition~2.7]{DHP:03}
that $-A_{\mathrm{CHS}}+\omega$ admits a bounded $\mathcal{H}^\infty$-calculus of $\mathcal{H}^\infty$-angle $\Phi^\infty_{A_{\mathrm{CHS}}+ \omega} < \pi/2$ for all $\omega > -\eta_0$. 

\begin{cor}\label{cor:hunendlichshift}
  Let $q \in [2,\infty)$ and $\eta_0$ as above. Then $-A_{\mathrm{CHS}}+ \omega$ admits a bounded $\mathcal{H}^\infty$-calculus on $\rX_m$ of $\mathcal{H}^\infty$-angle $\Phi^\infty_{A_{\mathrm{CHS}}+\omega} < \pi/2$ for
  all $\omega > -\eta_0$
\end{cor}

\autoref{cor:hunendlichshift} allows now to deduce maximal $\rLp$-$\rLq$-regularity estimates  for the linearized compressible system
\begin{equation}
	 	\left\{
	 	\begin{aligned}
                  \dt \zeta + \bar{\xi} \divH \bar{V}   &=f_1, &&\text{ on }G \times (0,\infty), \\
                  \dt  V  - \frac{\mu \DeltaH V }{(1-\delta z)\bar{\xi}}- \frac{\mu}{\bar{\xi}} \dz \big( \frac{1-\delta z}{\delta^2} \dz V \big ) -
                  \frac{\mu' \nablaH \divH V}{(1-\delta z)\bar{\xi}}  + \frac{\nablaH \zeta}{\bar{\xi}} &= f_2, &&\text{ on } \Omega \times (0,\infty),  \\
	 	\dz \zeta  &= 0  , &&\text{ on } \Omega \times (0,\infty), \\ 
	 	V &=0, &&\text{ on } \Gamma_u \times (0,\infty), \\
	 	\dz V &= 0 , &&\text{ on } \Gamma_b \times (0,\infty), \\
	 	(\zeta(0), V(0))  &=(\xi_0-\bar{\xi}, v_0)  .
	 	\end{aligned}
	 	\right. 
	 	\label{eq:commr}
	 \end{equation}
Note that the trace space $\rX_\gamma = (\rX_0,\rX_1)_{1-1/p}$ for the initial data $(\xi_0, v_0)$ is given by
\begin{equation*}
  \rX_\gamma \coloneqq \rH^{1,q}_\per(G;\R) \cap \rLq_0(G;\R) \times \rX_\gamma^v.
  \end{equation*}

It is now natural to decompose a given integrable function $f \colon G \to \R$ into 
\begin{equation*}
    f = f_m + f_{\mathrm{avg}}, \quad \text{ with } \quad \int_G f_m=0  \quad \text{ and } \quad f_{\mathrm{avg}} = \frac{1}{|G|}  \int_G f. 
  \end{equation*}

Integrating the first equation of \eqref{eq:commr}, we deduce that
	 \begin{equation*}
	 	\dt \zeta_{\mathrm{avg}} = f_{1, \mathrm{avg}} \text{ for } t>0 \quad\text { and } \quad \xi_{\mathrm{avg}}(0) =0.
              \end{equation*}
Thus  $ \zeta_{\mathrm{avg}} \in \rL^\infty(\R_+)$ provided  $f_{1,\mathrm{avg}} \in \rL^1(\R_+)$. In this case, the set of equations \eqref{eq:commr} is transformed into the set of equations
 \begin{equation}
	 	\left\{
	 	\begin{aligned}
                  \dt \zeta_m + \bar{\xi} \divH \bar{V}   &=f_1 - f_{1,\mathrm{avg}}, &&\text{ on }G \times (0,\infty), \\
                  \dt  V  - \frac{\mu \DeltaH V }{(1-\delta z)\bar{\xi}}- \frac{\mu}{\bar{\xi}} \dz \big ( \frac{1-\delta z}{\delta^2} \dz V \big ) -
                  \frac{\mu' \nablaH \divH V}{(1-\delta z)\bar{\xi}}  + \frac{\nablaH \zeta_m}{\bar{\xi}} &= f_2, &&\text{ on } \Omega \times (0,\infty),  \\
	 	\dz \zeta  &= 0  , &&\text{ on } \Omega \times (0,\infty), \\ 
	 	V &=0, &&\text{ on } \Gamma_u \times (0,\infty), \\
	 	\dz V &= 0 , &&\text{ on } \Gamma_b \times (0,\infty), \\
	 	(\zeta_m(0), V(0))  &=(\xi_0-\bar{\xi}, v_0)  .
	 	\end{aligned}
	 	\right. 
	 	\label{eq:mcommr}
	 \end{equation}

\begin{prop}\label{prop:CPEatmos}
 Let $\bar{\xi}>0$, $\eta_0 >0$  as above and $p,q \in [2,\infty)$ such that $1 \neq \frac{1}{p} + \frac{1}{2q}$ as well as $\frac{1}{2} \neq \frac{1}{p} + \frac{1}{2q}$. Moreover, assume that
 \begin{equation*}
 \begin{aligned}
 \mathrm{e}^{\eta(\cdot)} ( f_1, f_2) \in\E_{0,\infty}, \quad f_{1,\mathrm{avg}} \in \rL^1(\R_+)   \quad \text{ as well as }\quad (\xi_0, v_0) \in \rX_\gamma.
\end{aligned}
\end{equation*}
Then  the system \eqref{eq:mcommr} admits a unique, strong solution $(\zeta_m,v)$ for all $\eta \in (0,\eta_0)$ satisfying 
    \begin{equation*}
         \begin{aligned}
             \mathrm{e}^{\eta(\cdot)} ( \zeta_m , V) \in \E_{1,\infty}  \quad \text{ and }\quad  \zeta_{\mathrm{avg}}  \in \rL^\infty(\R_+).
         \end{aligned}
    \end{equation*}
Furthermore, there exists a constant $C>0$, depending only on $p$, $q$ and $\eta$, such that 
    \begin{equation}\label{eq:maxregest}
        \begin{aligned}
            &\|  \mathrm{e}^{\eta(\cdot)} ( \zeta_m , V) \|_{\E_{1,\infty} } + \| \zeta_{\mathrm{avg}} \|_{\rL^\infty(\R_+)} + \|   \mathrm{e}^{\eta(\cdot)} \dt \zeta_{\mathrm{avg}} \|_{\rL^\infty(\R_+)} \\
            &\leq C \big ( \| (\xi_0-\bar{\xi},v_0) \|_{\rX_\gamma} 
            + \|\mathrm{e}^{\eta(\cdot)} ( f_1, f_2) \|_{\E_{0,\infty} } + \|  f_{1,\mathrm{avg}} \|_{\rL^1(\R_+)}  \big ).
        \end{aligned}
    \end{equation}
\end{prop}

\begin{proof}
The proof is a direct consequence of \autoref{cor:hunendlichshift}. It yields  that the equation \eqref{eq:mcommr} admits a unique solution $(\zeta_m, V)$ satisfying \eqref{eq:maxregest}. 
\end{proof}

\subsection{Estimates  for  the nonlinear terms}\label{subsec:estnon}\mbox{}\\
In the following, we estimate the nonlinear terms given in \eqref{eq:nonlinearities2}. 
Given $\rX$ as in \autoref{subsec:lagrangecor}, we note that there is  a constant $C>0$ such that for all $\varepsilon \in ( 0, \frac{1}{2C})$ we have 
\begin{equation*}
    \begin{aligned}
        \| \nablaH \rX - \mathbb{I}_2\|_{\rL^\infty(\R_+ \times G)} &\leq C  \| \nablaH \rX - \mathbb{I}_2\|_{\rL^\infty(\R_+;\rH^{1,q}(G))} 
        \leq C \int_0^\infty \mathrm{e}^{-\eta t} \mathrm{e}^{\eta t} \| \nablaH \bar{V}(t,\cdot) \|_{\rH^{1,q}(G)} \d t \\
        &\leq C \big( \int_0^\infty   \mathrm{e}^{-p' \eta t} \d t \big)^{\nicefrac{1}{p'}} \|  \mathrm{e}^{\eta (\cdot )} V \|_{\E^v_{1,\infty}} 
        \leq C \varepsilon \leq \frac{1}{2}.
    \end{aligned}
\end{equation*}
Next, we conclude that $\det \nablaH \rX \geq C >0$ in $\R_+ \times G$ and obtain from \eqref{eq:ex of nablaH Y sea ice para-hyper} that 
$\|\mathrm{Z} \|_{\rL^\infty(\R_+; \rH^{1,q}(G))} \leq C$. Therefore,
\begin{equation*}
     \|\mathrm{Z}- \mathbb{I}_2 \|_{\rL^\infty(\R_+; \rH^{1,q}(G))} \leq    \|\mathrm{Z} \|_{\rL^\infty(\R_+; \rH^{1,q}(G))}  \|\nablaH \rX- \mathbb{I}_2 \|_{\rL^\infty(\R_+; \rH^{1,q}(G))} \leq C \varepsilon.
\end{equation*}
Notice that $\dt \nablaH \rX = \nablaH \bar{V}$ and hence 
\begin{equation*}
    \| \dt \nablaH \rX \|_{\rLp(\R_+;\rH^{1,q}(G))} \leq \| \mathrm{e}^{\eta (\cdot)} \nablaH \bar{V} \|_{\rLp(\R_+;\rH^{1,q}(G))} \leq C \varepsilon.
\end{equation*}
We also have 
\begin{equation*}
    \dt (\mathrm{Z}- \mathbb{I}_2 ) = \dt \mathrm{Z} = - (\nablaH \rX)^{-1} ( \dt \nablaH \rX ) (\nablaH \rX)^{-1} = -Z (\dt \nablaH \rX) Z,
\end{equation*}
and thus 
\begin{equation*}
    \| \dt (\mathrm{Z}- \mathbb{I}_2 ) \|_{\rLp(\R_+;\rH^{1,q}(G))} = \| \dt \mathrm{Z} \|_{\rLp(\R_+;\rH^{1,q}(G))} \leq C \varepsilon.
\end{equation*}

We fix now $\eta \in (0,\eta_0)$ and choose $\varepsilon$ so  small that the transform $\rX$ is indeed a $C^1$-diffeomorphism. Moreover, for $\varepsilon>0$ define 
\begin{equation}\label{def:ball}
  \mathbb{B}_\varepsilon \coloneqq \big \{(\zeta, V) \in \E_{1,\infty} \cap X_m \colon \| \mathrm{e}^{\eta(\cdot)}  (\zeta_m, V) \|_{\E_{1,\infty }}
  + \| \zeta_{\mathrm{avg}} \|_{\rL^\infty(\R_+)} + \|   \mathrm{e}^{\eta(\cdot)} \dt \zeta_{\mathrm{avg}} \|_{\rL^\infty(\R_+)}\leq \varepsilon \big \}.
    \end{equation}

We then obtain the following estimates for  the nonlinear terms.

\begin{prop}\label{prop:estnonlin}
Let $p,q \in (2,\infty)$. There exists a constant $C>0$,  depending only on $p$, $q$ and $\eta$, such that for all $\varepsilon$ sufficiently small and
all $(\zeta, V)  \in \mathbb{B}_{\varepsilon}$, the functions $F_1,F_2$ and $F_{1,\mathrm{avg}}$ satisfy
    \begin{equation*}
        \| \mathrm{e}^{\eta(\cdot)}(F_1,F_2)) \|_{\E_0,\infty} + \| F_{1,\mathrm{avg}}\|_{\rL^1(\R_+)} \leq C \varepsilon^2.
    \end{equation*}
  \end{prop}
  
\begin{proof}
We note that for every $(\zeta, V)  \in \mathbb{B}_{\varepsilon}$ we have
    \begin{equation}
        \label{eq:sigmaball}
        \|\zeta\|_{\rL^\infty(\R_+; \rH^{1,q}(G))} \leq \| \mathrm{e}^{\eta(\cdot)}\zeta_m \|_{\rL^\infty(\R_+; \rH^{1,q}(G))} + \| \zeta_{\mathrm{avg}} \|_{\rL^\infty(\R_+)} \leq C\varepsilon,
    \end{equation}
    as well as 
    \begin{equation}
        \label{eq:Vball}
        \| V \|_{\rL^\infty(\R_+; \rH^{1,q}(\Omega))} + \| \nabla V \|_{\rL^p(\R_+; \rL^\infty(\Omega))} \leq C \varepsilon,
    \end{equation}
    due to the embedding  $\E^v_1 \hookrightarrow \rL^\infty(\R_+;\rH^{1,q}(\Omega))$. Recalling  that 
    \begin{equation*}
        \mathrm{e}^{\eta t}F_1 = - \mathrm{e}^{\eta t}\zeta \divH \bar{V} -\mathrm{e}^{\eta t} (\zeta + \bar{\xi}) \nablaH \bar{V} : \left [\mathrm{Z}^\top  - \mathbb{I}_2 \right ],
    \end{equation*}
    and hence
    \begin{equation*}
        F_{1,\mathrm{avg}} = - \frac{1}{|G|} \int_G \zeta \divH \bar{V} \d (x,y) - \frac{1}{|G|} \int_G  (\zeta + \bar{\xi}) \nablaH \bar{V} : \left [\mathrm{Z}^\top  - \mathbb{I}_2 \right ] \d (x,y),
    \end{equation*}
we obtain
    \begin{align*}
        \|  \mathrm{e}^{\eta t}F_1 \|_{\rL^p(\R_+;\rH^{1,q}(G))} 
        &\leq \|  (\zeta + \bar{\xi}) \|_{\rL^\infty(\R_+; \rH^{1,q}(G))} \| \mathrm{Z}^\top - \mathbb{I}_2 \|_{\rL^\infty(\R_+; \rH^{1,q}(G))} \| \mathrm{e}^{\eta(\cdot)}\nablaH \bar{V} \||_{\rL^p(\R_+; \rH^{1,q}(G))} \\ & \quad + \| \zeta \|_{\rL^\infty(\R_+; \rH^{1,q}(G))} \| \mathrm{e}^{\eta(\cdot)} \divH \bar{V}\||_{\rL^p(\R_+; \rH^{1,q}(G))} \\
        &\leq C \varepsilon^2.
    \end{align*}
Next, we estimate the second term of $F_{1,\mathrm{avg}}$ as 
    \begin{align*}
      &\big \| \int_G  (\zeta + \bar{\xi}) \nablaH \bar{V} : \left [\mathrm{Z}^\top  - \mathbb{I}_2 \right ] \big \|_{\rL^1(\R_+)} \\
      &\le \| (\zeta + \bar{\xi}) \|_{\rL^\infty(\R_+ \times G)} \|\mathrm{Z}^\top  - \mathbb{I}_2 \|_{\rL^\infty(\R_+ \times G)} \int_0^\infty \int_G \nablaH \bar{V} \d (x,y) \d t \\
      &\le C  \|  (\zeta+ \bar{\xi}) \|_{\rL^\infty(\R_+; \rH^{1,q}(G))} \| \mathrm{Z}^\top - \mathbb{I}_2 \|_{\rL^\infty(\R_+; \rH^{1,q}(G))}
        \int_0^\infty \mathrm{e}^{-\eta t} \mathrm{e}^{\eta t} \| \nablaH \bar{V} \|_{\rL^q(G)} \d t \\
        &\le C \varepsilon ( \int_0^\infty \mathrm{e}^{-p' \eta t } \d t)^{\nicefrac{1}{p'}} \| \mathrm{e}^{\eta (\cdot)} \nablaH \bar{V} \|_{\rLp(\R_+;\rL^q(G))} \\
        &\le C \varepsilon^2.
    \end{align*}
 The first  term of $F_{1,\mathrm{avg}}$ can be estimated similarly. The same is true for $F_2$ since all nonlinear terms are at least quadratic in terms of $V$, $\zeta$ or $ \mathrm{Z}- \mathbb{I}_2$. We only show this in detail for the third term of the first line of $F_2$ and obtain  
 \begin{equation*}
        \begin{aligned}
          \Big \| \mathrm{e}^{\eta (\cdot )} \big (\sum_{l,j,k} \mathrm{Z}_{l,k} \frac{\partial V_i}{\partial y_k} \frac{\partial \mathrm{Z}_{k,j}}{\partial y_l} \big ) \Big \|_{\rLp(\R_+;\rLq(\Omega))}
          &=   \big \| \mathrm{e}^{\eta (\cdot )} \big (\sum_{l,j,k} \mathrm{Z}_{l,k} \frac{\partial V_i}{\partial y_k} \frac{\partial \left (\mathrm{Z}_{k,j}-\delta_{k,j} \right )}{\partial y_l} \big )
        \big \|_{\rLp(\R_+;\rLq(\Omega))} \\
            &\leq C\| \mathrm{Z} \|_{\rL^\infty (\R_+;\rL^{\infty}(G))}  \| \mathrm{e}^{\eta (\cdot )}  V\|_{\E^v_{1,\infty}}  \left \| \mathrm{Z}- \mathbb{I}_2 \right \|_{\rL^\infty (\R_+;\rH^{1,q}(G))} \\
            &\leq C \varepsilon^2. 
        \end{aligned}
    \end{equation*}
 Hence, 
    \begin{equation*}
        \| \mathrm{e}^{\eta(\cdot)}F_2 \|_{\rLp(\R_+;\rLq(\Omega))} \leq C \varepsilon^2.
    \end{equation*}
\end{proof}

We are finally  in position to give a proof of  \autoref{thm:globalWP}.

\begin{proof}[Proof of \autoref{thm:globalWP}]
Let $\mathbb{B}_\varepsilon$ be given as in \eqref{def:ball}.  For $(\zeta_1, V_1) \in  \mathbb{B}_\varepsilon$  denote by $(\zeta,V) \coloneqq \Psi(\zeta_1, V_1) \in\E_{1,\infty}$ the unique solution of
the linearized equation \eqref{eq:CPEatmostransformed} with right-hand sides $F_1(\zeta_1, V_1)$ and  $F_2(\zeta_1, V_1)$ guaranteed by  \autoref{prop:estnonlin}. We will  show
that $\Psi$ is a selfmap and a contraction. Choosing the initial data small enough, i.\ e.,
\begin{equation*}
    \| (\xi_0-\bar{\xi},v_0) \|_{\rX_\gamma} \leq \frac{\varepsilon}{2C},
\end{equation*}
where $C>0$ denotes the constant given in \autoref{prop:CPEatmos}. The estimates  \eqref{eq:maxregest}  and \autoref{prop:estnonlin} yield
 \begin{equation*}
        \begin{aligned}
            \|  \mathrm{e}^{\eta(\cdot)} ( \zeta_m , V) \|_{\E_{1,\infty} } &+ \| \zeta_{\mathrm{avg}} \|_{\rL^\infty(\R_+)} + \|   \mathrm{e}^{\eta(\cdot)} \dt( \zeta_{\mathrm{avg}}) \|_{\rL^\infty(\R_+)} \\
            &\leq C \big( \| (\xi_0-\bar{\xi},v_0) \|_{\rX_\gamma} 
            + \|\mathrm{e}^{\eta(\cdot)} ( F_1, F_2) \|_{\E_{0,\infty} } + \|F_{1,\mathrm{avg}} \|_{\rL^\infty(\R_+)}  \big ) \leq \frac{\varepsilon}{2} + C\varepsilon^2.
        \end{aligned}
    \end{equation*}
We conclude that for $\varepsilon>0$ small enough, $\Psi$ is a self-map, i.\, e., $\Psi \colon \mathbb{B}_\varepsilon \rightarrow \mathbb{B}_\varepsilon$. The contraction property follows in the same way. 

Finally, since $\varepsilon$ may be chosen to be sufficiently small, the transform $\rX(t,\cdot )$ is a $C^1$-diffeomorphism from $G$ into $G$ and we set $\mathrm{Y}(t,\cdot) = \rX^{-1}(t,\cdot)$. Defining 
\begin{equation*}
    \begin{aligned}
        \xi( t,x,y) := \zeta( t,\mathrm{Y} (t,x,y)) \quad \text{ and } \quad  v (t,x,y,z) &:= V (t,\mathrm{Y}(t,x,y),z)
    \end{aligned}
\end{equation*}
we verify that $(\xi,v)$ is a global, strong solution to system \eqref{eq:compressibleprimitivegravity2} satisfying  $(\xi,v) \in \E_{1,\infty}$.
\end{proof}

\begin{proof}[Proof of \autoref{thm:generalpressure}]
We will not provide a detailed proof of \autoref{thm:generalpressure} as its strategy  is similar to the one described in the proofs of \autoref{thm:localwpCAO} and \autoref{thm:globalWP}.
A short  discussion of the  main differences is nevertheless in order.  The hydrostatic Lagrange transformation yields the operator
$A = \frac{\mu}{\xi_0}\Delta + \frac{\mu'}{\xi_0}\nablaH \divH$ on $\rLq(\Omega;\R^2)$ and we show as in  \autoref{sec:lintheo} that $-A+\omega$ admits a
bounded $\mathcal{H}^\infty$-calculus on $\rL^q(\Omega;\R^2)$
of $\mathcal{H}^\infty$-angle strictly less than $\pi/2$.  The estimates on the nonlinear terms differ only by the pressure term. For the local well-posedness we estimate  the transformed pressure
$P'(\zeta)  (\mathrm{Z}^\top \nablaH\zeta)$ as in the proof of \autoref{lem:nonlinearestimates}.

To deduce the global well-posedness result for small data we rewrite the transformed pressure term as
\begin{equation*}
    P'(\zeta+ \bar{\xi})  (\mathrm{Z}^\top \nablaH\zeta) = \nablaH ( P(\zeta) ) + P'(\zeta+ \bar{\xi}) \sum_j \frac{\partial \zeta}{\partial y_j}(\mathrm{Z}_{j,i} - \delta_{j,i}).
\end{equation*}
Setting $\Tilde{\zeta} = P(\zeta)$, we solve for $\Tilde{\zeta}$ similarly  as in \autoref{thm:globalWP}. The pressure law \eqref{eq:pressurelaw} ensures then the invertibility of $P$
and hence for proving the assertion  it is sufficient to solve the resulting system for $\Tilde{\zeta}$. 
\end{proof}

\medskip 

{\bf Acknowledgements}
 \small{This research is supported by the Basque Government through the BERC 2022-2025 program and by the Spanish State Research Agency through BCAM Severo Ochoa excellence accreditation CEX2021-01142-S funded by MICIU/AEI/10.13039/501100011033 and through Grant PID2023-146764NB-I00 funded by MICIU/AEI/10.13039/501100011033 and cofunded by the European Union. Matthias Hieber acknowledges the support by DFG project FOR~5528. Arnab Roy would like to thank the Alexander von Humboldt-Foundation, Grant RYC2022-036183-I funded by MICIU/AEI/10.13039/501100011033 and by ESF+. Yoshiki Iida was supported by JSPS KAKENHI Grant Number JP24KJ2080, and partially supported by Top Global University Project of Waseda University.
He is also grateful to Fachbereich Mathematik, TU Darmstadt for its hospitality while visiting Professor Matthias Hieber.}


\begin{thebibliography}{99}
\bibitem{ABHN:11} W. Arendt, Ch. Batty, M. Hieber, F. Neubrander,
{\it Vector-Valued  Laplace Transforms and Cauchy Problems}. Monographs in Mathematics, vol. 96, $2^{nd}$-edition, Birkh\"auser, Basel, 2011. 
  
			
\bibitem{Ama:95}
H.~Amann,
{\it Linear and Quasilinear Parabolic Problems}. Monographs in Mathematics, vol.~89, Birkh\"auser, 1995.
			
\bibitem{Ama:19}
H.~Amann,
{\it Linear and Quasilinear Parabolic Problems. Vol.~II. Function Spaces}. Monographs in Mathematics, vol.~106, Birkh\"auser/Springer, Cham, 2019.

	
			
\bibitem{CT:07}
C.~Cao, E.S.~Titi, 
Global well-posedness of the three-dimensional viscous primitive equations of large scale ocean and atmosphere dynamics.
{\it Ann. of Math.} {\bf 166} (2007), 245--267. 
			
			

            \bibitem{DDHPV:04}
            R.~Denk, G.~Dore, M.~Hieber, J.~Pr\"uss, A.~Venni,
            New thoughts on old results of R. T. Seeley.
            {\it Math. Ann.} {\bf 328} (2004), 545--583.
            
	    \bibitem{DHP:03} 
            R.~Denk, M.~Hieber, J.~Pr\"uss,
            {\it $\mathcal{R}$-Boundedness, Fourier Multipliers and Problems of Elliptic and Parabolic Type}. 
            Mem. Amer. Math. Soc. {\bf 788} 2003.
			
	   \bibitem{DHP:07}
	    R.~Denk, M.~Hieber, J.~Pr\"uss, 
	  Optimal $\rLp$-$\rLq$-estimates for parabolic boundary value problems with inhomogeneous data.
	{\it Math. Z.} {\bf 257} (2007), 193--224.


      \bibitem{Dan:14}
        R. Danchin,
        A Lagrangian approach for the compressible Navier-Stokes equations.
        {\it  Ann. de l'Institut Fourier } {\bf 64} (2014), 753-791. 


      \bibitem{Dan:18}
        R. Danchin,
        Fourier analysis methods for the compressible Navier-Stokes equations.
        {\it Handbook of Math. Anal. in Mechanics of Viscous Fluids}, Y. Giga, A. Novotny (eds.), Springer, 2018, 1843--1903. 

        
      \bibitem{DM:23}
        R. Danchin, P. Mucha,
        Compressible Navier-Stokes equations with ripped density.
        {\it  Comm. Pure Appl. Math.} {\bf 76} (2023), 3437--3492.

        \bibitem{DN:13}
        R.~Denk, T.~Nau,
        Discrete Fourier multipliers and cylindrical boundary-value problems.
{\it Proc. Roy. Soc. Edinburgh} {\bf 143} (2013), 1163--1183.

\bibitem{ES:18}
Y.~Enomoto, Y.~Shibata, 
Global Existence of Classical Solutions and Optimal Decay Rate for Compressible Flows via the Theory of Semigroups. 
{\it Handbook of Math. Anal. in Mechanics of Viscous Fluids}, Y. Giga, A. Novotny (eds.), Springer, 2018, 2085--2181.

\bibitem{EN:12} M. Ersoy, T. Ngom, 
  Existence of a global weak solution to compressible primitive equations.
  {\it C.R. Acad Sci.} {\bf 350}, (2012), 379--382.

\bibitem{ENS:11} M. Ersoy, T. Ngom, M. Sy,
Compressible primitive equations: formal derivation and stability of weak solutions.
  {\it Nonlinearity} {\bf 24}, (2011), 79--96.

\bibitem{EF01}  
E. Feireisl, A. Novotn\'y{} and H. Petzeltov\'a, On the existence of globally defined weak solutions to the Navier-Stokes equations, {\it J. Math. Fluid Mech.} {\bf 3} (2001), no.~4, 358--392.

\bibitem{GGHHK:17}
Y.~Giga, M.~Gries, M.~Hieber, A.~Hussein, T.~Kashiwabara, 
Bounded $\Hinfty-$calculus for the hydrostatic Stokes operator on $\rL^p-$spaces and applications.
{\it Proc. Amer. Math. Soc.} {\bf 145} (2017), 3865--3876.
			
\bibitem{GGHHK:20b}
Y.~Giga, M.~Gries, M.~Hieber, A.~Hussein, T.~Kashiwabara,
Analyticity of solutions to the primitive equations.
{\it Math. Nachr.} {\bf 293} (2020), 284--304.
		
            \bibitem{HH:05}
            R.~Haller-Dintelmann, M.~Hieber,
            $\mathcal{H}^\infty$-calculus for products of non-commuting operators.
            {\it Math. Z.} {\bf 251} (2005), 85--100.
   
			\bibitem{HK:16}
			M.~Hieber, T.~Kashiwabara,
			Global strong well-posedness of the three dimensional primitive equations in $\rL^p$-spaces.
			{\it Arch. Rational Mech. Anal.} {\bf 221} (2016), 1077--1115. 
		

              


                      \bibitem{Hof:02}
                        D. Hoff,
                        Dynamics of singularity surfaces for compressible, viscous flows in two space dimensions.
                        {\it Comm. Pure Appl. Math.} {\bf 55} (2002), 1365--1407.



                      
			\bibitem{HNVW:16} 
			T.~Hyt\"onen, J.~van Neerven, M.~Veraar, L.~Weis,
			{\it Analysis in Banach spaces, vol.~II}, Springer, Cham, 2016.
			
		
			
                         \bibitem{KZ:07a}
                         I.~Kukavica, M.~Ziane,
                         On the regularity of the primitive equations of the ocean.
                         {\it Nonlinearity} {\bf 20} (2007), 2739--2753.
			
		
			
\bibitem{LTW:92a}
            J.-L.~Lions, R.~Temam, S.H.~Wang,
            New formulations of the primitive equations of atmosphere and applications.
            {\it Nonlinearity} {\bf 5} (1992), 237--288.

            \bibitem{LTW:92b}
            J.-L.~Lions, R.~Temam, S.H.~Wang,
            On the equations of the large-scale ocean.
            {\it Nonlinearity} {\bf 5} (1992), 1007--1053.
			
   
\bibitem{LTW:95}
J.-L.~Lions, R.~Temam, S.H.~Wang,
Mathematical theory for the coupled atmosphere-ocean models. (CAO III).
{\it J. Math. Pures Appl.~(9)}~{\bf 74} (1995), 105-163. 

\bibitem{PLL98}
P.-L. Lions, Mathematical topics in fluid mechanics. Vol. 2, {\it Oxford Lecture Series in Mathematics and its Applications, Oxford Science Publications}, 10 , Oxford Univ. Press, New York, 1998.
             

\bibitem{LT:weak}
X.~Liu, E.~Titi,
Global existence of weak solutions to the compressible primitive equations of atmospheric dynamics with degenerate viscosities.             
{\it SIAM J. Math. Anal.} {\bf 51} (2019), 1913--1964.  

\bibitem{LT:20}
X.~Liu, E.~Titi,
Zero Mach number limit of the compressible primitive equations part I: Well-prepared initial data.
{\it Arch. Rational Mech. Anal.} {\bf 238} (2020), 705--747. 


\bibitem{LT:21}
X.~Liu, E.~Titi,
Local well-posedness of strong solutions to the three-dimensional compressible primitive equations,
{\it Arch. Rational Mech. Anal.} {\bf 241} (2021), 729--764. 

\bibitem{MN:83}
A.~Matsumura, T.~Nishida,
Initial boundary value problems for the equations of motion of compressible viscous and heat-conductive fluids,
{\it Comm. Math. Phys.} {\bf  89} (1983), 445--464.

			\bibitem{N:12}
            T.~Nau,
            {\it $\rLp$-Theory of Cylindrical Boundary Value Problems}.
            Springer Spektrum, Wiesbaden, 2012, xvi+188 pp.
            
            
				
			\bibitem{PS:16} 
			J.~Pr\"uss, G.~Simonett,
			{\it Moving Interfaces and Quasilinear Parabolic Evolution Equations}. Monographs in Mathematics, vol. 105, Birkh\"auser, 2016.
			
	
			
	
			
			\bibitem{Tri:78}
			H.~Triebel,
			{\it Interpolation Theory, Function Spaces, Differential Operators}. North-Holland, 1978.
			

         
\end{thebibliography}
\end{document}